\documentclass[11pt,american,english]{elsarticle}	
\usepackage[T1]{fontenc}
\usepackage[latin9]{inputenc}
\usepackage[a4paper]{geometry}
\usepackage{babel}
\usepackage{textcomp}
\usepackage{amsthm}
\usepackage{amsmath}
\usepackage{amssymb}
\usepackage{setspace}
\usepackage{esint}
\usepackage{lineno,hyperref}	
\modulolinenumbers[5]	
\journal{Journal de Math\'ematiques Pures et Appliqu\'ees} 
\usepackage{breakurl}

\makeatletter
\theoremstyle{plain}
\newtheorem{thm}{\protect\theoremname}
  \theoremstyle{remark}
  \newtheorem{rem}[thm]{\protect\remarkname}
  \theoremstyle{plain}
  \newtheorem{prop}[thm]{\protect\propositionname}

\usepackage{multirow}

\usepackage{arydshln}

\usepackage{lmodern}
\newcommand{\1}{\mbox{1\hspace{-1mm}I}}
\numberwithin{equation}{section}

\usepackage{calc}
\makeatother

  \addto\captionsamerican{\renewcommand{\propositionname}{Proposition}}
  \addto\captionsamerican{\renewcommand{\remarkname}{Remark}}
  \addto\captionsamerican{\renewcommand{\theoremname}{Theorem}}
  \addto\captionsenglish{\renewcommand{\propositionname}{Proposition}}
  \addto\captionsenglish{\renewcommand{\remarkname}{Remark}}
  \addto\captionsenglish{\renewcommand{\theoremname}{Theorem}}
  \providecommand{\propositionname}{Proposition}
  \providecommand{\remarkname}{Remark}
\providecommand{\theoremname}{Theorem}

\begin{document}
\selectlanguage{american}%
\global\long\def\1{\mbox{1\hspace{-1mm}I}}

\selectlanguage{english}%

\begin{frontmatter}		
	
	\title{THE MASTER EQUATION IN MEAN FIELD THEORY}
	
	\author[AB]{Alain Bensoussan\fnref{A1}}
	\address[AB]{International Center for Decision and Risk Analysis\\Jindal School of Management, University of Texas at Dallas,\\College of Science and Engineering\\Systems Engineering and Engineering Management\\City University Hong Kong}
	\fntext[A1]{Corresponding author (Phone:+1 972-883-6117; Email:axb046100@utdallas.edu). 
	His research is supported
	by the National Science Foundation under grant DMS-1303775 and the
		Research Grants Council of the Hong Kong Special Administrative Region
		(CityU 500113). He also thanks P. Cardialaguet for providing his
		notes on the topic.}
	
	\author[JF]{Jens Frehse}
	\address[JF]{Institute for Applied Mathematics, University of Bonn}
	
	\author[PY]{Sheung Chi Phillip Yam\fnref{A3}}
	\address[PY]{Department of Statistics, The Chinese University of Hong Kong}
	\fntext[A3]{The third author acknowledges the financial supports from The Hong Kong RGC GRF 404012 with the project title: Advanced Topics In Multivariate Risk Management In Finance And Insurance, The Chinese University of Hong Kong Direct Grants 2011/2012 Project ID: 2060444.}
	
	\begin{abstract}
		In his lectures at {\it College de France}, P.L. Lions introduced the concept
		of Master equation, see \cite{PLL} for Mean Field Games. It is introduced
		in a heuristic fashion, from the prospective as a system of partial differential equations,
		that the equation is associated to a Nash equilibrium for a large, but finite, number of
		players. The method, also explained in \cite{PCA}, composed of a formalism of derivations. The interest of this equation is that it contains
		interesting particular cases, which can be studied directly, in particular
		the system of HJB-FP (Hamilton-Jacobi-Bellman, Fokker-Planck) equations
		obtained as the limit of the finite Nash equilibrium game, when the
		trajectories are independent, see \cite{LAL}. Usually, in mean field
		theory, one can bypass the large Nash equilibrium, by introducing
		the concept of representative agent, whose action is influenced by
		a distribution of similar agents, and obtains directly the system of
		HJB-FP equations of interest, see for instance \cite{BFY}. Apparently,
		there is no such approach for the Master equation. We show here that
		it is possible. We first do it for the Mean Field type control problem,
		for which we interpret completely the Master equation. For the Mean
		Field Games itself, we solve a related problem, and obtain again the
		Master equation.\\
		{\bf R\'esum\'e} 
		
		Dans son cours au Coll\`ege de France, P.L. Lions a inroduit le concept d'\'equation maitresse, pour les jeux \`a champ moyen \cite{PLL}.  Ceci a \'et\'e fait d'une mani\`ere heuristique, \`a partir du syst\`eme d'\'equations aux d\'eriv\'ees partielles associ\'e \`a un \'equilibre de Nash, pour un nombre fini, mais grand, de joueurs. La m\'ethode repose sur un formalisme de d\'erivations. L'int\'er\^et de cette \'equation maitresse est qu'elle contient des cas particuliers int\'eressants, qui peuvent \^etre \'etudi\'es dicretement, en particulier le syst\`eme des \'equations HJB-FP, Hamilton-Jacobi-Bellman \& Fokker Planck, obtenu comme la limite d'un \'equilibre de Nash, lorsque les trajectoires sont ind\'ependantes \cite{LAL}. De mani\`ere g\'en\'erale, dans la th\'eorie des jeux a champ moyen, on peut ne pas passer par l'\'equilibre de Nash pour un grand nombre de joueurs, en introduisant le concept d'agent repr\'esentatif, dont le comportement d\'epend d'une distribution d'agents similaires a l'agent repr\'esentatif \cite{BFY}. Une telle possibilit\'e n'avait pas \'et\'e mise en avant, pour l'\'equation maitresse. Nous montrons que c'est possible. Nouns le montrons d'abord pour les probl\`emes de contr\^ole de type champ moyen, et nous caract\'erisons compl\`etement l'\'equation maitresse. Pour les jeux champ moyen, nous r\'esolvons un probl\`eme relatif \`a ce cas, et obtenons \`a nouveau l'\'equation maitresse.
	\end{abstract}

	
\end{frontmatter}

\section{INTRODUCTION}

Since we do not intend to give complete proofs, we proceed formally, by
assuming relevant smoothness structure whenever it eases the argument. We consider functions
$f(x,m,v),\; g(x,m,v),\; h(x,m)$ and $\sigma(x)$ where $x\in \mathbb{R}^{n};$ $m$
is a probability measure on $\mathbb{R}^{n}$, but we shall retain ourselves mostly in
the regular case, in which $m$ represents the probability density,
assumed to be in $L^{2}(\mathbb{R}^{n})$ and $v$ is a control in $\mathbb{R}^{d}$. The functions $f$ and $h$ are scalar, $g$ is a vector in
$\mathbb{R}^{n}$ and $\sigma(x)$ is a $n\times n$ matrix. All these functions
are differentiable in all arguments. In the case of the differentiability with respect
to $m$, we use the concept of G\^ateaux differentiability. Indeed, $F:L^{2}(\mathbb{R}^{n})\rightarrow \mathbb{R}$ is said to be G\^ateaux differentiable if there uniquely exists $\frac{\partial F}{\partial m}(m)\in L^{2}(\mathbb{R}^{n})$, such that
\begin{equation*}
\frac{d}{d\theta}F(m+\theta\tilde{m})|_{\theta=0}=\int_{\mathbb{R}^{n}}\frac{\partial F}{\partial m}(m)(\xi)\tilde{m}(\xi)d\xi.
\end{equation*}
The second order derivative is a linear map from $L^{2}(\mathbb{R}^{n})$ into itself
, defined by 
\begin{equation*}
\frac{d}{d\theta}\frac{\partial F}{\partial m}(m+\theta\tilde{m})(\xi)|_{\theta=0}=\int_{\mathbb{R}^{n}}\frac{\partial^{2}F}{\partial m^{2}}(m)(\xi,\eta)\tilde{m}(\eta)d\eta.
\end{equation*}

We can state the second order Taylor's formula as

\[
F(m+\tilde{m})=F(m)+\int_{\mathbb{R}^{n}}\frac{\partial F}{\partial m}(m)(\xi)\tilde{m}(\xi)d\xi+\int_{\mathbb{R}^{n}}\int_{\mathbb{R}^{n}}\frac{\partial^{2}F}{\partial m^{2}}(m)(\xi,\eta)\tilde{m}(\xi)\tilde{m}(\eta)d\xi d\eta.
\]

Consider a probability space $(\Omega,\mathcal{A},\mathbb{P})$ on which various Wiener processes are
defined. We define first a standard Wiener
process $w(t)$ in $\mathbb{R}^{n}$. To avoid cumbersome of notation, if there is no ambiguity, we shall suppress the arguments of functions in the rest of this paper.

We first introduce the classical mean field type control problem in which the state dynamic is represented by the stochastic differential equation of McKean-Vlasov type:
\begin{equation}
\label{eq:1.1}
\left\{
\begin{split}
	&dx=g(x,m_{v(\cdot)},v(x))dt+\sigma(x)dw,\\
	&x(0)=x_{0};\\
\end{split}
\right. 
\end{equation}
in which $v(x)$ is a feedback and $m_{v(\cdot)}(x,t)$ is the probability density of the state $x(t)$. The initial value $x_{0}$ is a random variable independent of the Wiener process $w(\cdot)$. This density is well-defined provided that there is invertibility of $a(x)=\sigma(x)\sigma^{*}(x)$. We define the second order differential operator 

\[
\mathcal{A}\varphi(x)=-\frac{1}{2}\sum_{i,j}a_{ij}(x)\frac{\partial^{2}\varphi(x)}{\partial x_{i}\partial x_{j}}
\]
 and its adjoint 

\[
\mathcal{A}^{*}\varphi(x)=-\frac{1}{2}\sum_{i,j}\frac{\partial^{2}(a_{ij}(x)\varphi(x))}{\partial x_{i}\partial x_{j}}.
\]

The mean field type control problem is to minimize the cost functional
\begin{equation}
\label{eq:1.2}
J(v(\cdot))= \mathbb{E} \bigg[\int_{0}^{T}f(x(t),m_{v(\cdot)}(t),v(x(t)))dt+h(x(T),m_{v(\cdot)}(T))\bigg].
\end{equation}

In the classical mean field games problem, one should first fix $m(\cdot)\in C([0,T];L^{2}(\mathbb{R}^{n}))$ as a given parameter in the state equation 
\begin{equation}
\label{eq:1.3}
\left\{
\begin{array}{rcl}
	dx&=&g(x,m,v(x))dt+\sigma(x)dw,\\
	x(0) & =&x_{0},\\ 
\end{array}
\right.
\end{equation}
and the objective functional
\begin{equation}
\label{eq:1.4}
J(v(\cdot),m(\cdot))= \mathbb{E} \bigg[\int_{0}^{T}f(x(t),m(t),v(x(t)))dt+h(x(T),m(T))\bigg].
\end{equation}

The mean field games problem looks for an equilibrium $\hat{v}(\cdot),m(\cdot)$ such that
\begin{equation}
\left\{
\begin{array}{l}
J(\hat{v}(\cdot),m(\cdot)) \leq J(v(\cdot),m(\cdot)),\:\forall v(\cdot),\\
 m(t)  \text{ is the probability density of }\hat{x}(t),\,\forall t\in[0,T],\label{eq:1.5}
\end{array}
\right.
\end{equation}
where $\hat{x}(\cdot)$ is the solution of (\ref{eq:1.3}) corresponding to the equilibrium pair $\hat{v}(\cdot),m(\cdot)$.

\section{MASTER EQUATION FOR THE CLASSICAL CASE $ $}

We refer to problems (\ref{eq:1.1}),(\ref{eq:1.2}) and (\ref{eq:1.3}),(\ref{eq:1.4}),(\ref{eq:1.5})
as the classical case. We define the Hamiltonian $H(x,m,q):\, \mathbb{R}^{n}\times L^{2}(\mathbb{R}^{n})\times \mathbb{R}^{n}\rightarrow \mathbb{R}$
as 
\begin{equation*}
H(x,m,q)=\inf_{v}(f(x,m,v)+q\cdot g(x,m,v))
\end{equation*}
and the optimal value of $v$ is denoted by $\hat{v}(x,m,q).$ We
then set 
\begin{equation*}
G(x,m,q)=g(x,m,\hat{v}(x,m,q)).
\end{equation*}
\subsection{MEAN FIELD TYPE CONTROL PROBLEM}

The mean field type control problem is easily transformed into a control problem in which the state is a probability density process $m_{v(\cdot)}$, which satisfies the solution of the Fokker-Planck equation 
\begin{equation}
\left\{
\begin{split}
&\frac{\partial m_{v(\cdot)}}{\partial t}+\mathcal{A}^{*}m_{v(\cdot)}+\text{div }(g(x,m_{v(\cdot)},v(x))m_{v(\cdot)}(x))=0\label{eq:2.1},\\
&m_{v(\cdot)}(x,0)=m_{0}(x). 
\end{split}
\right.
\end{equation}

Here $m_{0}(x)$ is the density of the initial value
$x_{0}.$ The objective functional $J(v(\cdot))$ can be written as 
\begin{equation}
J(v(\cdot))=\int_{0}^{T}\int_{\mathbb{R}^{n}}f(x,m_{v(\cdot)}(t),v(x))m_{v(\cdot)}(x,t)dxdt+\int_{\mathbb{R}^{n}}h(x,m_{v(\cdot)}(T))m_{v(\cdot)}(x,T)dx\label{eq:2.2}
\end{equation}

We next use the traditional invariant embedding approach. Define a
family of control problems indexed by initial conditions $(m,t)$:
\begin{equation}
\label{eq:2.3}
\left\{
\begin{array}{l}
	\dfrac{\partial m_{v(\cdot)}}{\partial s}+\mathcal{A}^{*}m_{v(\cdot)}+\text{div }(g(x,m_{v(\cdot)},v(x))m_{v(\cdot)}(x))=0,\\
	m_v(x,t)=m(x);
\end{array}
\right. 
\end{equation}
\begin{equation}
J_{m,t}(v(\cdot))=\int_{t}^{T}\int_{\mathbb{R}^{n}}f(x,m_{v(\cdot)}(s),v(x))m_{v(\cdot)}(x,s)dxds+\int_{\mathbb{R}^{n}}h(x,m_{v(\cdot)}(T))m_{v(\cdot)}(x,T)dx,\label{eq:2.4}
\end{equation}
and we set 
\begin{equation}
V(m,t)=\inf_{v(\cdot)}J_{m,t}(v(\cdot))\label{eq:2.5}.
\end{equation}

We can then write the Dynamic Programming equation satisfied by $V(m,t).$
By standard arguments, one obtains \footnote{This equation has also been obtained, independently, by M. Lauriere
and O. Pironneau \cite{LPI} %
} 
\begin{equation}
\label{eq:2.6}
\left\{
\begin{array}{l}
	\dfrac{\partial V}{\partial t}-\displaystyle\int_{\mathbb{R}^{n}}\dfrac{\partial V(m)}{\partial m}(\xi)\mathcal{A}^{*}m(\xi)d\xi\\
	\qquad+\displaystyle\inf_{v}\left(\displaystyle\int_{\mathbb{R}^{n}}f(\xi,m,v(\xi))m(\xi)d\xi\right.\left.-\displaystyle\int_{\mathbb{R}^{n}}\dfrac{\partial V(m)}{\partial m}(\xi)\text{div }(g(\xi,m,v(\xi))m(\xi))d\xi\right)=0,\\
	V(m,T)=\int_{\mathbb{R}^{n}}h(x,m)m(x)dx.
\end{array}
\right.
\end{equation}

By setting
\begin{equation*}
U(x,m,t)=\frac{\partial V(m,t)}{\partial m}(x),
\end{equation*}
we can rewrite (\ref{eq:2.6}) as 
\begin{equation}
\frac{\partial V}{\partial t}-\int_{\mathbb{R}^{n}}\mathcal{A}U(\xi,m,t)m(\xi)d\xi+\int_{\mathbb{R}^{n}}H(\xi,m,DU)m(\xi)d\xi=0\label{eq:2.9}
\end{equation}
since the optimization in $v$ can be done inside the integral. We
next differentiate (\ref{eq:2.9}) with respect to $m$ which gives
\begin{equation}
\begin{array}{l}
\dfrac{\partial}{\partial m}\bigg[\displaystyle\int_{\mathbb{R}^{n}}H(\xi,m,DU(\xi,m,t))m(\xi)d\xi\bigg](x)=H(x,m,DU(x))\\
\qquad+\displaystyle\int_{\mathbb{R}^{n}}\frac{\partial}{\partial m}H(\xi,m,DU(\xi))(x)m(\xi)d\xi+\int_{\mathbb{R}^{n}}G(\xi,m,DU(\xi))m(\xi)D_{\xi}\frac{\partial}{\partial m}U(\xi,m,t)(x)d\xi.
\end{array}
\end{equation}

Hence under sufficient continuous differentiability, using 
\begin{equation}
\label{add_3}
\frac{\partial}{\partial m}U(\xi,m,t)(x)=\frac{\partial}{\partial m}U(x,m,t)(\xi)=\frac{\partial^{2}V(m,t)}{\partial m^{2}}(x,\xi),
\end{equation}
we obtain the master equation
\begin{equation}
\label{eq:2.10}
\left\{
\begin{array}{l}
	-\dfrac{\partial U}{\partial t}+\mathcal{A}U+\displaystyle\int_{\mathbb{R}^{n}}\frac{\partial}{\partial m}U(x,m,t)(\xi)(\mathcal{A}^{*}m(\xi)+\text{div (}G(\xi,m,DU(\xi))m(\xi))d\xi\\
	\qquad=H(x,m,DU(x))+\displaystyle\int_{\mathbb{R}^{n}}\frac{\partial}{\partial m}H(\xi,m,DU(\xi))(x)m(\xi)d\xi,\\
	U(x,m,T)=h(x,m)+\displaystyle\int_{\mathbb{R}^{n}}\frac{\partial}{\partial m}h(\xi,m)(x)m(\xi)d\xi.
\end{array}
\right.
\end{equation}

The probability density, corresponding to the optimal feedback control, is given by
\begin{equation*}
\left\{
\begin{array}{l}
	\dfrac{\partial m}{\partial t}+\mathcal{A}^{*}m+\text{div }(G(x,m,DU)m(x))=0,\\
	m(x,0)=m_{0}(x).
\end{array}
\right.
\end{equation*}

Define $u(x,t)=U(x,m(t),t),$ then clearly, from (\ref{eq:2.10})
we obtain 
\begin{equation}
\label{eq:2.12}
\left\{
\begin{split}
	-\dfrac{\partial u}{\partial t}+\mathcal{A}u  &=H(x,m,Du(x))+\displaystyle\int_{\mathbb{R}^{n}}\frac{\partial}{\partial m}H(\xi,m,Du(\xi))(x)m(\xi)d\xi,\\
	u(x,T)&=h(x,m)+\displaystyle\int_{\mathbb{R}^{n}}\frac{\partial}{\partial m}h(\xi,m)(x)m(\xi)d\xi;	
\end{split}
\right.
\end{equation}
which together with the FP equation

\begin{equation}
\label{eq:2.14}
\left\{
\begin{array}{l}
\dfrac{\partial m}{\partial t}+\mathcal{A}^{*}m+\text{div }(G(x,m,Du)m(x)) =0,\\
m(x,0)=m_{0}(x);
\end{array}
\right.
\end{equation}
form the system of coupled HJB- FP equations of the classical mean
field type control problem, see \cite{BFY}. 

\subsection{\label{sub:MEAN-FIELD-GAMES}MEAN FIELD GAMES }

In Mean Field Games, we cannot have a Bellman equation, similar to
(\ref{eq:2.6}), (\ref{eq:2.7-1}), since the problem is not simply a control problem. However, if we first fixed parameter $m(\cdot)$ in (\ref{eq:1.3}), (\ref{eq:1.4}) we have a standard control problem. We introduce the state dynamics and the cost functional accordingly
\begin{equation}
\left\{
\begin{split}
dx(s)&=g(x(s),m,v(x(s)))ds+\sigma(x(s))dw,\label{eq:1.3-1}\\
x(t) &=x;
\end{split}
\right.
\end{equation}
\begin{equation}
J_{x,t}(v(\cdot),m(\cdot))= \mathbb{E} \bigg[\int_{t}^{T}f(x(s),m(s),v(x(s)))ds+h(x(T),m(T))\bigg].\label{eq:1.4-1}
\end{equation}

If we set
\[
u(x,t)=\inf_{v(\cdot)}J_{x,t}(v(\cdot),m(\cdot))
\]
in which we omit to write explicitly the dependence of $u$ in $m.$ Then
$u(x,t)$ satisfies Bellman equation 
\begin{equation}
\left\{
\begin{split}
-\dfrac{\partial u}{\partial t}+\mathcal{A}u&=H(x,m,Du(x)),\label{eq:2.12-1}\\
u(x,T)&=h(x,m).
\end{split}
\right.
\end{equation}

For mean field games, we next require that $m$ must be the probability density of the optimal state, hence 
\begin{equation}
\left\{
\begin{array}{l}
\dfrac{\partial m}{\partial t}+\mathcal{A}^{*}m+\text{div }(G(x,m,Du)m(x)) =0,\label{eq:2.14-1}\\
m(x,0)=m_{0}(x);
\end{array}
\right.
\end{equation}
 and this is the system of HJB-FP equations, corresponding to the
classical Mean Field Games problem. We can check that, if one considers
the Master equation
\begin{equation}
\label{eq:2.10-1}
\left\{
\begin{array}{l}
-\dfrac{\partial U}{\partial t}+\mathcal{A}U+\displaystyle\int_{\mathbb{R}^{n}}\frac{\partial}{\partial m}U(x,m,t)(\xi)(\mathcal{A}^{*}m(\xi)+\text{div (}G(\xi,m,DU(\xi))m(\xi))d\xi\\
\qquad=H(x,m,DU(x)),\\
U(x,m,T)=h(x,m);
\end{array}
\right.
\end{equation}
then 
\begin{equation}
u(x,t)=U(x,m(t),t)\label{eq:2.12-2}.
\end{equation}
Combining (\ref{eq:2.10-1}) and (\ref{eq:2.14-1}), we obtain easily (\ref{eq:2.12-1}), and hence (\ref{eq:2.12-2}) if the problem is well-posed.

\section{STOCHASTIC MEAN FIELD TYPE CONTROL}

\subsection{\label{sub:PRELIMINARIES}PRELIMINARIES}

If we look at the formulation (\ref{eq:2.1}), (\ref{eq:2.2}) of
the mean field type control problem, it is a deterministic problem,
although at the origin it was a stochastic one, see (\ref{eq:1.1}),
(\ref{eq:1.2}). We now consider a stochastic version of (\ref{eq:2.1}),
(\ref{eq:2.2}) or a doubly stochastic version of (\ref{eq:1.1}),
(\ref{eq:1.2}). Let us begin with this one. Assume that there is a second
standard Wiener process $b(t)$ with values in $\mathbb{R}^{n}$; $b(t)$ and
$w(t)$ are independent which are also independent of $x_{0}$. We set $\mathcal{B}^{t}$=$\sigma(b(s):s\leq t)$
and $\mathcal{F}$$^{t}$=$\sigma(x_{0},b(s),w(s):\: s\leq t)$ .
The control $v(x,t)$ at time $t$ is a feedback, but not deterministic, i.e. the functional form of $v$ is $\mathcal{B}^{t}$ adapted. We consider
the stochastic McKean-Vlasov equation 
\begin{equation}
\label{eq:3.1}
\left\{
\begin{split}
dx&=g(x,m_{v(\cdot)}(t),v(x,t))dt+\sigma(x)dw+\beta db(t),\\
x(0)&=x_{0},\\
\end{split}
\right. 
\end{equation}
in which $m_{v(\cdot)}(t)$ represents the conditional probability density of $x(t)$, given the $\sigma$-algebra $\mathcal{B}^{t}$ . The stochastic mean field type control problem aims at minimizing the objective functional
\begin{equation}
J(v(\cdot))= \mathbb{E} \bigg[\int_{0}^{T}f(x(t),m_{v(\cdot)}(t),v(x(t),t))dt+h(x(T),m_{v(\cdot)}(T))\bigg].\label{eq:3.2}
\end{equation}
\subsection{CONDITIONAL PROBABILITY}

Let $y(t)=x(t)- \beta b(t)$, then the process $y(t)$ satisfies the equation 
\begin{equation}
\label{eq:3.1-1}
\left\{
\begin{split}
&dy=g(y(t)+\beta b(t),m_{v(\cdot)}(t),v(y(t)+\beta b(t),t))dt+\sigma(y(t)+\beta b(t))dw,\\
&y(0)=x_{0}.\\
\end{split}
\right.
\end{equation}
If we fix $b(s),s\leq t$, then the conditional probability of $y(t)$
is simply the probability density arising from the Wiener process
$w(t)$, in view of the independence of $w(t)$ and $b(t).$ It is
the function $p(y,t)$ solution of 

\begin{equation}
\left\{
\begin{array}{l}
	\dfrac{\partial p}{\partial t}-\dfrac{1}{2}\displaystyle\sum_{ij}(a_{i,j}(y+\beta b(t))\frac{\partial^{2}p}{\partial y_{i}\partial y_{j}})+\text{div }\left(g(y+\beta b(t),m_{v(\cdot)}(t),v(y+\beta b(t),t))p\right)=0\\
	p(y,0)=m_{0}(y)\\
\end{array}
\right.
\end{equation}

The conditional probability density of $x(t)$ given $\mathcal{B}^{t}$
is $m_{v(\cdot)}(x,t)=p(x-\beta b(t),t)$, and hence 
\[
\partial_{t}m_{v(\cdot)}=(\frac{\partial p}{\partial t}+\frac{1}{2}\beta^{2}\Delta p)(x-\beta b(t),t)dt-\beta Dp(x-\beta b(t),t)db(t).
\]

We thus have
\begin{equation}
\label{eq:3.2-1}
\left\{
\begin{array}{l}
\partial_{t}m_{v(\cdot)}+(\mathcal{A}^{*}m_{v(\cdot)}-\dfrac{1}{2}\beta^{2}\Delta m_{v(\cdot)}+\text{div}(g(x,m_{v(\cdot)}(t),v(x,t))m_{v(\cdot)}))dt+\beta Dm_{v(\cdot)}db(t)=0,\\
m_{v(\cdot)}(x,0)=m_{0}(x);
\end{array}
\right.
\end{equation}
and the objective functional (\ref{eq:3.2}) can be written as 
\begin{equation}
J(v(\cdot))= \mathbb{E} \bigg[\int_{0}^{T}\int_{\mathbb{R}^{n}}f(x,m_{v(\cdot)}(t),v(x,t))m_{v(\cdot)}(x,t)dxdt+\int_{\mathbb{R}^{n}}h(x,m_{v(\cdot)}(T))m_{v(\cdot)}(x,T)dx\bigg].\label{eq:3.3}
\end{equation}
The problem becomes a stochastic control problem for a distribution-valued parameter system. Using the invariant embedding again, we consider the family of problems indexed by $m,t$
\begin{equation}
\label{eq:3.4}
\left\{
\begin{array}{l}
\partial_{s}m_{v(\cdot)}+(\mathcal{A}^{*}m_{v(\cdot)}-\dfrac{1}{2}\beta^{2}\Delta m_{v(\cdot)}+\text{div}(g(x,m_{v(\cdot)}(s),v(x,s))m_{v(\cdot)}))ds+\beta Dm_{v(\cdot)}db(s)=0,\\
m_{v(\cdot)}(x,t)=m(x),\\
\end{array}
\right.
\end{equation}
and 
\begin{equation}
J_{m,t}(v(\cdot))= \mathbb{E} \bigg[\int_{t}^{T}\int_{\mathbb{R}^{n}}f(x,m_{v(\cdot)}(s),v(x,s))m_{v(\cdot)}(x,s)dxdt+\int_{\mathbb{R}^{n}}h(x,m_{v(\cdot)}(T))m_{v(\cdot)}(x,T)dx\bigg].\label{eq:3.5}
\end{equation}

Set 
\[
V(m,t)=\inf_{v(\cdot)}J_{m,t}(v(\cdot)),
\]
then $V(m,t)$ satisfies the Dynamic Programming equation 
\begin{equation}
\label{eq:2.6-1}
\left\{
\begin{array}{l}
	\dfrac{\partial V}{\partial t}-\displaystyle\int_{\mathbb{R}^{n}}\dfrac{\partial V(m,t)}{\partial m}(\xi)(\mathcal{A}^{*}m(\xi)-\frac{1}{2}\beta^{2}\Delta m(\xi))d\xi  \\
	\qquad+\dfrac{1}{2}\beta^{2}\displaystyle\int_{\mathbb{R}^{n}}\displaystyle\int_{\mathbb{R}^{n}}\dfrac{\partial^{2}V(m,t)}{\partial m^{2}}(\xi,\eta)Dm(\xi)Dm(\eta)d\xi d\eta\\
	\qquad+\displaystyle\inf_{v}\bigg(\displaystyle\int_{\mathbb{R}^{n}}f(\xi,m,v(\xi))m(\xi)d\xi-\displaystyle\int_{\mathbb{R}^{n}}\dfrac{\partial V(m,t)}{\partial m}(\xi)\text{div }(g(\xi,m,v(\xi))m(\xi))d\xi\bigg)=0,\\
	V(m,T)=\displaystyle\int_{\mathbb{R}^{n}}h(x,m)m(x)dx.
\end{array}
\right.
\end{equation}
\subsection{MASTER EQUATION }

To obtain the master equation, we define again $U(x,m,t)=\frac{\partial V(m,t)}{\partial m}(x)$ and we have the terminal condition
\begin{equation}
U(x,m,T)=h(x,m)+\int_{\mathbb{R}^{n}}\frac{\partial h(\xi,m)}{\partial m}(x)m(\xi)d\xi,\label{eq:2.6-2}
\end{equation}
and from (\ref{eq:2.6-1}) we obtain
\begin{equation}
\label{eq:2.7-1}
\begin{array}{l}
	\dfrac{\partial V}{\partial t}-\displaystyle\int_{\mathbb{R}^{n}}(\mathcal{A}U-\dfrac{1}{2}\beta^{2}\Delta U)(\xi,m,t)m(\xi)d\xi\\
	\qquad+\dfrac{1}{2}\beta^{2}\displaystyle\int_{\mathbb{R}^{n}}\displaystyle\int_{\mathbb{R}^{n}}\dfrac{\partial U(\xi,m,t)}{\partial m}(\eta)Dm(\xi)Dm(\eta)d\xi d\eta+\displaystyle\int_{\mathbb{R}^{n}}H(\xi,m,DU)m(\xi)d\xi=0.\\
\end{array}
\end{equation}

We then differentiate (\ref{eq:2.7-1}) with respect to $m$ to obtain the master equation. We note that
\begin{equation*}
\begin{array}{rcl}
&&\dfrac{\partial}{\partial m}\left(\displaystyle\int_{\mathbb{R}^{n}}\displaystyle\int_{\mathbb{R}^{n}}\dfrac{\partial U(\xi,m,t)}{\partial m}(\eta)Dm(\xi)Dm(\eta)d\xi d\eta\right)(x)\\
&=&\displaystyle\int_{\mathbb{R}^{n}}\displaystyle\int_{\mathbb{R}^{n}}\dfrac{\partial^{2}U(x,m,t)}{\partial m^{2}}(\xi,\eta)Dm(\xi)Dm(\eta)d\xi d\eta-2\text{div}\left(\displaystyle\int_{\mathbb{R}^{n}}\dfrac{\partial U(x,m,t)}{\partial m}(\eta)Dm(\eta)d\eta\right),
\end{array}
\end{equation*}
and 
\begin{equation*}
\begin{array}{rcl}
&&\dfrac{\partial}{\partial m}\left(\displaystyle\int_{\mathbb{R}^{n}}H(\xi,m,DU)m(\xi)dx\right)(x)\\
&=&H(x,m,DU(x))+\displaystyle\int_{\mathbb{R}^{n}}\dfrac{\partial}{\partial m}H(\xi,m,DU(\xi))(x)m(\xi)d\xi\\
&&\qquad +\displaystyle\int_{\mathbb{R}^{n}}\dfrac{\partial}{\partial m}(DU(\xi,m,t))(x)G(\xi,m,DU(\xi))m(\xi)d\xi.
\end{array}
\end{equation*}

Next, using (\ref{add_3}) in terms of those in the present setting, we have 
\[
\int_{\mathbb{R}^{n}}\frac{\partial}{\partial m}(DU(\xi,m,t))(x)G(\xi,m,DU(\xi))m(\xi)d\xi=\int_{\mathbb{R}^{n}}D_{\xi}(\frac{\partial}{\partial m}U(\xi,m,t)(x))G(\xi,m,DU(\xi))m(\xi)d\xi
\]
\[
=-\int_{\mathbb{R}^{n}}\frac{\partial}{\partial m}U(\xi,m,t)(x)\:\text{div}(G(\xi,m,DU(\xi))m(\xi))d\xi=-\int_{\mathbb{R}^{n}}\frac{\partial}{\partial m}U(x,m,t)(\xi)\:\text{div}(G(\xi,m,DU(\xi))m(\xi))d\xi.
\]

Collecting results, we obtain the Master equation 
\begin{equation}
\label{eq:2.8-1}
\left\{
\begin{array}{l}
	-\dfrac{\partial U}{\partial t}+\mathcal{A}U-\dfrac{1}{2}\beta^{2}\Delta U\\
	\qquad+\displaystyle\int_{\mathbb{R}^{n}}\dfrac{\partial}{\partial m}U(x,m,t)(\xi)\left(\mathcal{A}^{*}m(\xi)-\dfrac{1}{2}\beta^{2}\Delta m(\xi)+\text{div (}G(\xi,m,DU(\xi))m(\xi))\right)d\xi\\
	\qquad-\dfrac{1}{2}\beta^{2}\displaystyle\int_{\mathbb{R}^{n}}\displaystyle\int_{\mathbb{R}^{n}}\dfrac{\partial^{2}U(x,m,t)}{\partial m^{2}}(\xi,\eta)Dm(\xi)Dm(\eta)d\xi d\eta+\beta^{2}\text{div}\left(\displaystyle\int_{\mathbb{R}^{n}}\dfrac{\partial U(x,m,t)}{\partial m}(\xi)Dm(\xi)d\xi\right)\\
	=H(x,m,DU(x))+\displaystyle\int_{\mathbb{R}^{n}}\dfrac{\partial}{\partial m}H(\xi,m,DU(\xi))(x)m(\xi)d\xi;\\
	U(x,m,T)=h(x,m)+\displaystyle\int_{\mathbb{R}^{n}}\dfrac{\partial h(\xi,m)}{\partial m}(x)m(\xi)d\xi.
\end{array}
\right.
\end{equation}

We note that this equation reduces to (\ref{eq:2.10}) when $\beta=0$.
\subsection{SYSTEM OF HJB-FP EQUATIONS }

We first check that we can derive from the Master equation a system
of coupled stochastic HJB-FP equations. Consider the conditional probability density process corresponding to the optimal feedback $\hat{v}(x,m,DU(x,m,t))$, which is the solution of 
\begin{equation}
\left\{
\begin{array}{l}
\partial_{t}m+(\mathcal{A}^{*}m-\dfrac{1}{2}\beta^{2}\Delta m+\text{div}(G(x,m,DU)m))dt+\beta Dm.db(t)=0\label{eq:3.6}\\
m(x,0)=m_{0}(x).\\
\end{array}
\right.
\end{equation}

Set $u(x,t)=U(x,m(t),t)$, we obtain 
\begin{equation}
\label{eq:3.7}
\left\{
\begin{array}{l}
	-\partial_{t}u+(\mathcal{A}u-\dfrac{1}{2}\beta^{2}\Delta u)dt +\beta^{2}\text{div}\left(\displaystyle\int_{\mathbb{R}^{n}}\dfrac{\partial U(x,m,t)}{\partial m}(\xi)Dm(\xi)d\xi\right)dt\\
	\qquad=\bigg(H(x,m,Du(x))+\displaystyle\int_{\mathbb{R}^{n}}\dfrac{\partial}{\partial m}H(\xi,m,Du(\xi))(x)m(\xi)d\xi\bigg)dt+\beta\displaystyle\int_{\mathbb{R}^{n}}\dfrac{\partial U(x,m,t)}{\partial m}(\xi)Dm(\xi)d\xi db(t)\\
	u(x,T)=h(x,m)+\displaystyle\int_{\mathbb{R}^{n}}\dfrac{\partial h(\xi,m)}{\partial m}(x)m(\xi)d\xi.
\end{array}
\right.
\end{equation}

Let
\begin{equation}
	B(x,t)=\int_{\mathbb{R}^{n}}\frac{\partial U(x,m,t)}{\partial m}(\xi)Dm(\xi)d\xi,
\end{equation}
we can rewrite (\ref{eq:3.6}), (\ref{eq:3.7}) as follows, by noting the It\^o's correction term of $u$ that involves the second derivative of $U$ with respect to $m$,
\begin{equation}
\label{eq:3.7-1}
\begin{array}{l}
\left\{
\begin{array}{l}
-\partial_{t}u+(\mathcal{A}u-\dfrac{1}{2}\beta^{2}\Delta u)dt +\beta^{2}\text{div}B(x,t)dt\\
\qquad=\bigg(H(x,m,Du(x))+\displaystyle\int_{\mathbb{R}^{n}}\dfrac{\partial}{\partial m}H(\xi,m,Du(\xi))(x)m(\xi)d\xi\bigg)dt+\beta B(x,t) db(t)\\
u(x,T)=h(x,m)+\displaystyle\int_{\mathbb{R}^{n}}\dfrac{\partial h(\xi,m)}{\partial m}(x)m(\xi)d\xi.
\end{array}
\right.\\
\\
\left\{
\begin{array}{l}
\partial_{t}m+(\mathcal{A}^{*}m-\dfrac{1}{2}\beta^{2}\Delta m+\text{div}(G(x,m,Du)m))dt+\beta Dmdb(t)=0\\
m(x,0)=m_{0}(x).\\
\end{array}
\right.
\end{array}
\end{equation}

Since the equation for $u$ is a backward stochastic partial differential equation (BSPDE), the solution is expressed by the pair $(u(x,t),B(x,t))$ which is adapted to the filtration
$\mathcal{B}^{t}$.

\subsection{OBTAINING THE SYSTEM OF STOCHASTIC HJB-FP EQUATIONS BY CALCULUS OF
VARIATIONS}

In this section, we are going to check that the system (\ref{eq:3.7-1}) can be also obtained by calculus
of variations techniques, without referring to the Master equation.
This is similar to approach in the deterministic case, see \cite{BFY}.
We go back to the formulation (\ref{eq:3.2-1}), (\ref{eq:3.3}) .
Let $\hat{v}(x,t)$ be an optimal feedback (it is a random field
adapted to $\mathcal{B}^{t}$). We denote $m(x,t)=m_{\hat{v}(\cdot)}(x,t)$,
which is therefore the solution of 
\begin{equation}
\left\{
\begin{array}{l}
\partial_{t}m+(\mathcal{A}^{*}m-\dfrac{1}{2}\beta^{2}\Delta m+\text{div}(g(x,m(t),\hat{v}(x,t))m))dt+\beta Dmdb(t)=0,\label{eq:3.4-1}\\
m(x,0)=m_{0}(x).
\end{array}
\right.
\end{equation}

We can compute its G\^ateaux differential 
\[
\tilde{m}(x,t)=\dfrac{d}{d\theta}m_{\hat{v}(\cdot)+\theta v(\cdot)}(x,t)|_{\theta=0},
\]
which satisfies
\begin{equation}
\label{eq:3.4-2}
\left\{
\begin{array}{l}
	\partial_{t}\tilde{m}+(\mathcal{A}^{*}\tilde{m}-\dfrac{1}{2}\beta^{2}\Delta\tilde{m}+\text{div}(g(x,m(t),\hat{v}(x,t))\tilde{m}))dt+\beta D\tilde{m}db(t)\\
	\qquad+\text{div}\:\left(\Big[\displaystyle\int_{\mathbb{R}^{n}}\dfrac{\partial g(x,m,\hat{v}(x))}{\partial m}(\xi)\tilde{m}(\xi)d\xi+\dfrac{\partial g}{\partial v}(x,m,\hat{v}(x))v(x)\Big]m(x)\right)dt=0\\
	\tilde{m}(x,0)=0\\
\end{array}
\right.
\end{equation}

We next compute the G\^ateaux differential of the cost functional 
\begin{equation}
\begin{array}{rl}
	&\dfrac{d}{d\theta}J(\hat{v}(\cdot)+\theta v(\cdot))|_{\theta=0}\\
	=& \mathbb{E} \bigg[\displaystyle\int_{0}^{T}\int_{\mathbb{R}^{n}}f(x,m(t),\hat{v}(x,t))\tilde{m}(x,t)dxdt+\int_{0}^{T}\int_{\mathbb{R}^{n}}\frac{\partial f}{\partial v}(x,m(t),\hat{v}(x,t))v(x)m(x)dxdt\\	
	&\qquad+\displaystyle\int_{0}^{T}\int_{\mathbb{R}^{n}}\int_{\mathbb{R}^{n}}\frac{\partial f}{\partial m}(x,m(t),\hat{v}(x,t))(\xi)\tilde{m}(\xi,t)m(x,t)dxd\xi dt\\
	&\qquad+\displaystyle\int_{\mathbb{R}^{n}}h(x,m(T))\tilde{m}(x,T)dx+\int_{\mathbb{R}^{n}}\int_{\mathbb{R}^{n}}\frac{\partial}{\partial m}h(x,m(T))(\xi)\tilde{m}(\xi,T)m(x,T)dx\bigg].
\end{array}
\end{equation}

As an adjoint equation, consider the BSPDE:
\begin{equation}
\label{eq:3.9}
\left\{
\begin{array}{rcl}
	-\partial_{t}u&=&\bigg(f(x,m,\hat{v}(x))+Dug(x,m,\hat{v}(x))-\mathcal{A}u+\frac{1}{2}\beta^{2}\Delta u-\beta^{2}\text{div} B(x,t)\\
	&&\qquad+\displaystyle\int_{\mathbb{R}^{n}}\frac{\partial}{\partial m}f(\xi,m,\hat{v}(\xi))(x)m(\xi)d\xi+\int_{\mathbb{R}^{n}}Du(\xi)\frac{\partial}{\partial m}g(\xi,m,\hat{v}(\xi))(x)m(\xi)d\xi\bigg)dt\\
	&&+\beta B(x,t)db(t),\\
	u(x,T) & =&h(x,m)+\displaystyle\int_{\mathbb{R}^{n}}\frac{\partial}{\partial m}h(\xi,m)(x)m(\xi)d\xi.
\end{array}
\right.
\end{equation}

We can then write 
\begin{equation*}
\begin{array}{rcl}
	&&\displaystyle\frac{d}{d\theta}J(\hat{v}(\cdot)+\theta v(\cdot))|_{\theta=0}\\
	&=&\mathbb{E} \displaystyle\int_{0}^{T}\bigg[\int_{\mathbb{R}^{n}}\frac{\partial f}{\partial v}(x,m(t),\hat{v}(x,t))v(x)m(x)dx\\
	&&\qquad\qquad+\displaystyle\int_{\mathbb{R}^{n}}\tilde{m}(x,t)\Big[-\partial_{t}u+(\mathcal{A}u-\frac{1}{2}\beta^{2}\Delta u+\beta^{2}\text{div} B)-Du . g(x,m,\hat{v}(x))\\
	&&\qquad\qquad\qquad\qquad\qquad\qquad\qquad\qquad-\displaystyle\int_{\mathbb{R}^{n}}Du(\xi)\frac{\partial}{\partial m}g(\xi,m,\hat{v}(\xi))(x)m(\xi)d\xi\Big]dx\bigg]dt\\
	&&\qquad+\mathbb{E}\displaystyle\int_{\mathbb{R}^{n}}u(x,T)\tilde{m}(x,T) dx.
\end{array}
\end{equation*}

But 
\begin{equation*}
	\partial_{t}\int_{\mathbb{R}^{n}}u(x,t)\tilde{m}(x,t)dx=\int_{\mathbb{R}^{n}}\partial_{t}u(x,t)\tilde{m}(x,t)dx+\int_{\mathbb{R}^{n}}\partial_{t}\tilde{m}(x,t)u(x,t)dx+\beta^{2}\int_{\mathbb{R}^{n}}D\tilde{m}(x,t)B(x,t)dx,
\end{equation*}
therefore 
\[
\partial_{t}\int_{\mathbb{R}^{n}}u(x,t)\tilde{m}(x,t)dx-\int_{\mathbb{R}^{n}}\partial_{t}u(x,t)\tilde{m}(x,t)dx+\beta^{2}\int_{\mathbb{R}^{n}}\text{div}\: B(x,t)\,\tilde{m}(x,t)dx=\int_{\mathbb{R}^{n}}\partial_{t}\tilde{m}(x,t)u(x,t)dx,
\]
which implies
\begin{equation*}
\begin{array}{rcl}
	&&\displaystyle\frac{d}{d\theta}J(\hat{v}(\cdot)+\theta v(\cdot))|_{\theta=0}\\
	&=& \mathbb{E} \displaystyle\int_{0}^{T}\int_{\mathbb{R}^{n}}\frac{\partial f}{\partial v}(x,m(t),\hat{v}(x,t))v(x)m(x)dxdt\\
	&&\qquad+\mathbb{E} \displaystyle\int_{0}^{T}\int_{\mathbb{R}^{n}}u(x,t)\Big[\partial_{t}\tilde{m}(x,t)+(\mathcal{A}^{*}\tilde{m}-\frac{1}{2}\beta^{2}\Delta\tilde{m}+\text{div}(g(x,m(t),\hat{v}(x,t))\tilde{m}))dt\\
	&&\qquad+\text{div}\,(m(x)\displaystyle\int_{\mathbb{R}^{n}}\frac{\partial g(x,m,\hat{v}(x))}{\partial m}(\xi)\tilde{m}(\xi)d\xi)\: dt\Big]dx.
\end{array}
\end{equation*}

From (\ref{eq:3.4-2}) we obtain 
\begin{equation*}
\begin{array}{rcl}
	&&\displaystyle\frac{d}{d\theta}J(\hat{v}(\cdot)+\theta v(\cdot))|_{\theta=0}\\
	&=&\mathbb{E} \displaystyle\int_{0}^{T}\int_{\mathbb{R}^{n}}\frac{\partial f}{\partial v}(x,m(t),\hat{v}(x,t))v(x)m(x)dxdt\\
	&&\qquad- \mathbb{E} \displaystyle\int_{0}^{T}\int_{\mathbb{R}^{n}}  u(x,t)\text{div (}\frac{\partial g}{\partial v}(x,m,\hat{v}(x))v(x)m(x))dxdt\\
	&=&\mathbb{E} \displaystyle\int_{0}^{T}\int_{\mathbb{R}^{n}}(\frac{\partial f}{\partial v}(x,m(t),\hat{v}(x,t))v(x)+Du\frac{\partial g}{\partial v}(x,m,\hat{v}(x,t))v(x))m(x)dxdt,
\end{array}
\end{equation*}
and since $v$ is arbitrary, for if the support of $m$ is $\mathbb{R}^n$, $\hat{v}(x,t)$ satisfies (because of optimality) 
\[
\frac{\partial f}{\partial v}(x,m(t),\hat{v}(x,t))+Du\frac{\partial g}{\partial v}(x,m,\hat{v}(x,t))=0,
\]
and therefore $\hat{v}(x,t)=\hat{v}(x,m(t),Du(x,t))$. 

One obtains immediately that the pair $(u(x,t),m(x,t))$ is a solution
of the system of stochastic HJB-FP equations (\ref{eq:3.7-1}).

\section{STOCHASTIC MEAN FIELD GAMES}

\subsection{GENERAL COMMENTS}

How to obtain a system of coupled HJB-FP equations in the case of
stochastic mean field games? In the deterministic case, see Section
\ref{sub:MEAN-FIELD-GAMES}, the idea was to consider $m(t)$ as a
parameter, and to solve a standard stochastic control problem. For
this problem, one can use standard Dynamic Programming to obtain
the HJB equation, depending on the parameter $m(t).$ One expresses
next the fixed point property, namely that $m(t)$ is the probability
density of the optimal state. This leads to the FP equation. 

By analogy with what was done in the framework of mean field type
control, the HJB equation becomes stochastic, and $m(t)$ becomes
the conditional probability density. This motivates the model we develop
in this section

\subsection{THE MODEL }

Adopt the notation of Section \ref{sub:PRELIMINARIES}. We recall
that the state equation is the solution of 
\begin{equation}
\left\{
\begin{split}
	dx & =g(x,m(t),v(x,t))dt+\sigma(x)dw+\beta db(t)\label{eq:4.1},\\
	x(0) & =x_{0};
\end{split}
\right.
\end{equation}
and $\mathcal{B}^{t}$=$\sigma(b(s),s\leq t)$ , $\mathcal{F}$$^{t}$=$\sigma(x_{0},b(s),w(s),\: s\leq t)$.
This time $m(t)$ is a given process adapted to $\mathcal{B}^{t}$
with values in $L^{2}(\mathbb{R}^{n}).$ We again consider feedback controls
which are field processes adapted to the filtration $\mathcal{B}^{t}$. 

We follow the theory developed by Shi-ge Peng \cite{SGP}. We define
the function 
\begin{equation}
u(x,t)=\inf_{v(\cdot)}\:  \mathbb{E} ^{\mathcal{B}^{t}}[\int_{t}^{T}f(x(s),m(s),v(x(s),s))ds+h(x(T),m(T))]\label{eq:4.2}
\end{equation}
and show that it satisfies a stochastic Hamilton-Jacobi-Bellman equation.
Although, it is possible to proceed directly and formally from the
definition of $u(x,t),$ to find out the stochastic HJB equation,
the best is to postulate the equation and to use a verification argument.
The equation is 
\begin{equation}
\left\{
\begin{array}{l}
	-\partial_{t}u+(\mathcal{A}u-\frac{1}{2}\beta^{2}\Delta u)dt  +\beta^{2}\text{div} B(x,t)dt=H(x,m(t),Du(x))dt+\beta B(x,t)db(t),\label{eq:4.3}\\
	u(x,T) =h(x,m(T)).
\end{array}
\right.
\end{equation}

If we can solve (\ref{eq:4.3}) for the pair $(u(x,t),\: B(x,t))$, $\mathcal{B}^{t}$-adapted field processes, then the verification argument is as follows: for any $\mathcal{B}^{t}$-adapted field process $v(x,t)$, we write
\[
H(x,m(t),Du(x,t))\leq f(x,m(t),v(x,t))+Du(x,t)g(x,m(t),v(x,t)).
\]

 Consider the process $x(t)$, the solution of the state equation (\ref{eq:4.1}), we compute the Ito differential $d\, u(x(t),t)$ by using a generalized It\^o's formula due to Kunita, which gives
\begin{equation*}
\begin{array}{l}
	du(x(t),t)=(Du g(x(t),m(t),v(x(t),t))+\displaystyle\frac{1}{2}\text{tr}(a(x(t))D^{2}u)+\frac{\beta^{2}}{2}\Delta u)dt\\
	\qquad\qquad\qquad\qquad+Du(\sigma(x(t))dw+\beta db(t))+\partial_{t}u(x(t),t)-\beta^{2}\text{div}(B(x(t),t))dt.
\end{array}
\end{equation*}
Hence we have the inequality, by using (\ref{eq:4.3}),
\begin{equation*}
\begin{array}{l}
	h(x(T),m(T))+\displaystyle\int_{t}^{T}f(x(s),m(s),v(x(s),s))ds\\
	\qquad\qquad\geq u(x,t)+\displaystyle\int_{t}^{T}Du(x(s))(\sigma(x(s))dw(s)+\beta db(s))-\beta\int_{t}^{T}B(x(s))db(s),
\end{array}
\end{equation*}
 from which we get 
\[
u(x,t)\leq  \mathbb{E} ^{\mathcal{B}^{t}}[\int_{t}^{T}f(x(s),m(s),v(x(s),s))ds+h(x(T),m(T))],
\]
 and a similar development used for the optimal feedback yields the
equality, hence the property (\ref{eq:4.2}) follows. 

Next, consider the optimal feedback $\hat{v}(x,m(t),Du(x,t))$ and
impose the fixed point property that $m(t)$ is conditional probability
density of the optimal state, we get the stochastic FP equation 
\begin{equation}
\left\{
\begin{array}{l}
\partial_{t}m+(\mathcal{A}^{*}m-\frac{1}{2}\beta^{2}\Delta m+\text{div}(G(x,m,Du)m))dt+\beta Dmdb(t)=0,\label{eq:4.5}\\
m(x,0)=m_{0}(x).
\end{array}
\right.
\end{equation}

We thus have obtained the pair of HJB-FP equations for the stochastic mean field game problem, (\ref{eq:4.3}) and (\ref{eq:4.5}).

\subsection{THE MASTER EQUATION}

We shall derive the Master equation by writing $u(x,t)=U(x,m(t),t)$. We must have, by using (\ref{eq:4.5}), 
\begin{equation*}
\begin{array}{l}
	\partial_{t}u=\bigg[\displaystyle\frac{\partial U}{\partial t}-\int_{\mathbb{R}^{n}}\frac{\partial U(x,m(t),t)}{\partial m}(\xi)(\mathcal{A}^{*}m(\xi)-\frac{1}{2}\beta^{2}\Delta m(\xi)+\text{div }(G(\xi,m(t),DU(\xi))m(\xi)))d\xi\\
	\qquad\qquad\qquad\qquad+\displaystyle\frac{1}{2}\beta^{2}\int_{\mathbb{R}^{n}}\int_{\mathbb{R}^{n}}\frac{\partial^{2}U(x,m(t),t)}{\partial m^{2}}(\xi,\eta)Dm(\xi)Dm(\eta)d\xi d\eta\bigg]dt\\
	\qquad\qquad-\beta\displaystyle\int_{\mathbb{R}^{n}}\frac{\partial U(x,m(t),t)}{\partial m}(\xi)Dm(\xi)d\xi db(t).
\end{array}	
\end{equation*}

Comparing terms with (\ref{eq:4.3}), and letting
\begin{equation}
B(x,t)=\int_{\mathbb{R}^{n}}\frac{\partial U(x,m(t),t)}{\partial m}(\xi)Dm(\xi)d\xi,
\end{equation}
we have the master equation
\begin{equation}
\label{eq:4.7}
\left\{
\begin{array}{l}
	-\displaystyle\frac{\partial U}{\partial t}+\mathcal{A}U-\frac{1}{2}\beta^{2}\Delta U+\int_{\mathbb{R}^{n}}\frac{\partial U(x,m,t)}{\partial m}(\xi)\Big(\mathcal{A}^{*}m(\xi)-\frac{1}{2}\beta^{2}\Delta m(\xi)+\text{div}(G(\xi,m,DU(\xi))m(\xi))\Big)d\xi\\
	\qquad\qquad+\beta^{2}\text{div}\:(\displaystyle\int_{\mathbb{R}^{n}}\frac{\partial U(x,m,t)}{\partial m}(\xi)Dm(\xi)d\xi)-\frac{1}{2}\beta^{2}\int_{\mathbb{R}^{n}}\int_{\mathbb{R}^{n}}\frac{\partial^{2}U(x,m,t)}{\partial m^{2}}(\xi,\eta)Dm(\xi)Dm(\eta)d\xi d\eta\\
	\qquad=H(x,m,DU(x)),\\
	U(x,m,T)=h(x,m).\\	
\end{array}
\right.
\end{equation}
 
\begin{rem}
\label{Rem1}It is very important to notice that the property 
\begin{equation}
\frac{\partial U(x,m,t)}{\partial m}(\xi)=\frac{\partial U(\xi,m,t)}{\partial m}(x)\label{eq:4.8}
\end{equation}
is true for the Master equation of the Mean Field type control problem (\ref{eq:2.8-1}) and not true for the Master equation
of the Mean Field games problem, (\ref{eq:4.7}). We shall notice
this discrepancy in the case of linear quadratic (LQ) problems, in which
explicit formulas can be obtained. 
\end{rem}

\subsection{REMARKS ON THE APPROACH OF CARMONA-DELARUE \cite{CAD}}

Independently of our work, Carmona and Delarue \cite{CAD} have put
recently on Arxiv, an article in which they try to interpret the Master
equation differently. To synthesize our approach, we state that there
is a Bellman equation for the Mean field type control problem, in
which the space variable is in an infinite dimensional space. The
Master equation is obtained by taking the gradient in this space variable.
This differentiation is taken in the sense of Frechet. It is also
possible to derive from the Master equation a coupled system of stochastic
HJB-FP equations, in which the solution of the HJB equation appears
not as a value function, but as an adjoint equation for an infinite
dimentional stochastic control problem. This system can be also obtained
directly, and the Master equation is then a way to decouple the coupled
equations. In the case of mean field games, there is no Bellman equation,
so we introduce the system of HJB-FP equations directly, with a fixed
point approach. The Master equation is then obtained by trying to
decouple the HJB equation from the Fokker Planck equation. Carmona
and Delarue try to derive the Master equation from a common optimality
principle of Dynamic Programming, with constraints. The differences
between the Mean field games and Mean field type control cases stem
from the way the minimization is performed. It results that for the
Mean field type control, our Master equation is different from that
of Carmona-Delarue. To some extent, our approach is more analytic
(sufficient conditions) whereas the approach of Carmona-Delarue
is more probabilistic (necessary conditions). In particular, Carmona-Delarue
rely on the lift-up approach introduced by P.L. Lions in his lectures.
This method connects functions of random variables to functions of
probability measures. Both works are largely formal, and a full comparison
is a daunting work, which is left for future work. It may be interesting to also combine these ideas.

\section{LINEAR QUADRATIC PROBLEMS}

\subsection{ASSUMPTIONS AND GENERAL COMMENTS}

For linear quadratic problems, we know that explicit formulas can
be obtained. We shall see that we can solve explicitly the Master
equation in both cases, and recover all results obtained in the LQ
case. Symmetric and nonsymmetric Riccati equations naturally come
in the present framework. It is important to keep in mind that in
the Master equation, the argument $m$ is a general element of $L^{2}(\mathbb{R}^{n})$
and not necessarily a probability. When we use the value arising from
the FP equation, we get of course a probability density, provided that
the initial condition is a probability density. To synthesize formulas,
we shall use the following notation. If $m\in L^{1}(\mathbb{R}^{n})\times L^{2}(\mathbb{R}^{n})$
then we set 

\begin{equation}
m_{1}=\int_{\mathbb{R}^{n}}m(\xi)d\xi,\: y=\int_{\mathbb{R}^{n}}\xi m(\xi)d\xi.\label{eq:5.1}
\end{equation}

We then take 
\begin{equation}
\label{eq:5.2}
\begin{split}
f(x,m,v)&=\frac{1}{2}[x^{*}Qx+v^{*}Rv+(x-Sy)^{*}\bar{Q}(x-Sy)];\\
g(x,m,v)&=Ax+\bar{A}y+Bv;\\
h(x,m)&=\frac{1}{2}[x^{*}Q_{T}x+(x-S_{T}y)^{*}\bar{Q}_{T}(x-S_{T}y)];\\
\sigma(x)&=\sigma;\\
\end{split}
\end{equation}
and hence $a(x)=a=\sigma\sigma^{*}$. We deduce easily 
\begin{equation}
\label{eq:5.6}
H(x,m,q) =\frac{1}{2}x^{*}(Q+\bar{Q})\, x-x^{*}\bar{Q}Sy+\frac{1}{2}y^{*}S^{*}\bar{Q}Sy- \frac{1}{2}q^{*}BR^{-1}B^{*}q+q^{*}(Ax+\bar{A}y)
\end{equation}
and
\begin{equation}
G(x,m,q)=Ax+\bar{A}y-BR^{-1}B^{*}q.\label{eq:5.7}
\end{equation}

\subsection{MEAN FIELD TYPE CONTROL MASTER EQUATION}

We begin with Bellman equation (\ref{eq:2.7-1})
\begin{equation}
\label{eq:5.8}
\left\{
\begin{array}{l}
	\dfrac{\partial V}{\partial t}-\displaystyle\int_{\mathbb{R}^{n}}(\mathcal{A}U-\dfrac{1}{2}\beta^{2}\Delta U)(\xi,m,t)m(\xi)d\xi\\
	\qquad+\dfrac{1}{2}\beta^{2}\displaystyle\int_{\mathbb{R}^{n}}\displaystyle\int_{\mathbb{R}^{n}}\dfrac{\partial U(\xi,m,t)}{\partial m}(\eta)Dm(\xi)Dm(\eta)d\xi d\eta+\displaystyle\int_{\mathbb{R}^{n}}H(\xi,m,DU)m(\xi)d\xi=0.\\
	V(m,T)=\displaystyle\int_{\mathbb{R}^{n}}h(\xi,m)m(\xi)d\xi;
\end{array}
\right.
\end{equation}
in which 
\begin{equation}
U(\xi,m,t)=\frac{\partial V(m,t)}{\partial m}(\xi).\label{eq:5.10}
\end{equation}

We can rewrite (\ref{eq:5.8}) under the present linear quadratic setting
\begin{equation}
\label{eq:5.11}
\left\{
\begin{array}{l}
	\displaystyle\frac{\partial V}{\partial t}+\int_{\mathbb{R}^{n}}(\frac{1}{2}\text{tr}\, aD_{\xi}^{2}U(\xi,m,t)+\frac{1}{2}\beta^{2}\Delta_{\xi}U(\xi,m,t))m(\xi)d\xi\\
	\qquad+\displaystyle\frac{1}{2}\beta^{2}\int_{\mathbb{R}^{n}}\int_{\mathbb{R}^{n}}\sum_{i=1}^{n}\frac{\partial}{\partial\xi_{i}}\frac{\partial}{\partial\eta_{i}}\Big(\frac{\partial U(\xi,m,t)}{\partial m}(\eta)\Big)m(\xi)m(\eta)d\xi d\eta\\
	\qquad+\displaystyle\int_{\mathbb{R}^{n}}\bigg[\frac{1}{2}\xi^{*}(Q+\bar{Q})\xi-\xi^{*}\bar{Q}Sy+\frac{1}{2}y^{*}S^{*}\bar{Q}Sy\\
	\qquad\qquad\qquad-\displaystyle\frac{1}{2}(DU(\xi,m,t))^{*}BR^{-1}B^{*}DU(\xi,m,t)+(DU(\xi,m,t))^{*}(A\xi+\bar{A}y)\bigg]m(\xi)d\xi=0\\
	V(m,T)=\displaystyle\int_{\mathbb{R}^{n}}\frac{1}{2}\xi^{*}(Q_{T}+\bar{Q}_{T})\xi m(\xi)d\xi-\frac{1}{2}y^{*}(S_{T}^{*}\bar{Q}_{T}+\bar{Q}_{T}S_{T})y+\frac{1}{2}y^{*}S_{T}^{*}\bar{Q}_{T}S_{T}y\, m_{1}.
\end{array}
\right.
\end{equation}

We look for a solution $V$ in (\ref{eq:5.11}) of the form in which $P(t)$ and $\Sigma(t,m_1)$ are symmetric matrices
\begin{equation}
V(m,t)=\frac{1}{2}\int_{\mathbb{R}^{n}}\xi^{*}P(t)\xi m(\xi)d\xi+\frac{1}{2}y^{*}\Sigma(t,m_{1})y+\lambda(t,m_{1}).\label{eq:5.13}
\end{equation}

Clearly, the terminal condition of (\ref{eq:5.11}) yields that
\begin{align}
P(T) & =Q_{T}+\bar{Q}_{T},\;\Sigma(T,m_{1})=S_{T}^{*}\bar{Q}_{T}S_{T}\, m_{1}-(S_{T}^{*}\bar{Q}_{T}+\bar{Q}_{T}S_{T}),\:\lambda(T,m_{1})=0.\label{eq:5.14}
\end{align}

Next 
\begin{equation}
U(\xi,m,t)=\frac{1}{2}\xi^{*}P(t)\xi+y^{*}\Sigma(t,m_{1})\xi+\frac{1}{2}y^{*}\frac{\partial\Sigma(t,m_{1})}{\partial m_{1}}y+\frac{\partial\lambda(t,m_{1})}{\partial m_{1}},\label{eq:5.15}
\end{equation}
and hence
\begin{equation*}
\begin{array}{c}
	D_{\xi}U(\xi,m,t)=P(t)\xi+\Sigma(t,m_{1})y,\\ D_{\xi}^{2}U(\xi,m,t)=P(t),\\
	\displaystyle\frac{\partial U(\xi,m,t)}{\partial m}(\eta)=\eta^{*}\Sigma(t,m_{1})\xi+y^{*}\frac{\partial\Sigma(t,m_{1})}{\partial m_{1}}(\xi+\eta)+\frac{1}{2}y^{*}\frac{\partial^{2}\Sigma(t,m_{1})}{\partial m_{1}^{2}}y+\frac{\partial^{2}\lambda(t,m_{1})}{\partial m_{1}^{2}}.
\end{array}
\end{equation*}

We can see that the property (\ref{eq:4.8}) is satisfied, thanks
to the symmetry of the matrix $\Sigma(t,m_{1}).$ We need 
\[
\sum_{i=1}^{n}\frac{\partial}{\partial\xi_{i}}\frac{\partial}{\partial\eta_{i}}\frac{\partial U(\xi,m,t)}{\partial m}(\eta)=\text{tr }\Sigma(t,m_{1}).
\]

With these calculations, we can proceed on equation (\ref{eq:5.11})
and obtain 
\begin{equation*}
\begin{array}{l}
	\displaystyle\frac{1}{2}\int_{\mathbb{R}^{n}}\xi^{*}\frac{d}{dt}P(t)\xi m(\xi)d\xi+\frac{1}{2}y^{*}\frac{\partial}{\partial t}\Sigma(t,m_{1})y+\frac{\partial\lambda(t,m_{1})}{\partial t}\\
	\qquad+\displaystyle(\frac{1}{2}\text{tr}\, aP(t)+\frac{\beta^{2}}{2}\text{tr}P(t))m_{1}+\frac{1}{2}\beta^{2}\text{tr}\Sigma(t,m_{1})(m_{1})^{2}+\displaystyle\int_{\mathbb{R}^{n}}\bigg[\frac{1}{2}\xi^{*}(Q+\bar{Q})\xi-\xi^{*}\bar{Q}Sy+\frac{1}{2}y^{*}S^{*}\bar{Q}Sy\\
	\qquad-\displaystyle\frac{1}{2}(P(t)\xi+\Sigma(t,m_{1})y)^{*}BR^{-1}B^{*}(P(t)\xi+\Sigma(t,m_{1})y)+(P(t)\xi+\Sigma(t,m_{1})y)^{*}(A\xi+\bar{A}y)\bigg]m(\xi)d\xi=0.
	
\end{array}
\end{equation*}

We can identify terms and obtain 
\[
\frac{\partial\lambda(t,m_{1})}{\partial t}+(\frac{1}{2}\text{tr}\, aP(t)+\frac{\beta^{2}}{2}\text{tr}P(t))m_{1}+\frac{1}{2}\beta^{2}\text{tr}\Sigma(t,m_{1})(m_{1})^{2}=0.
\]

Therefore, from the final condition $\lambda(T,m_{1})=0,$ it follows that
\begin{equation}
\lambda(t,m_{1})=\int_{t}^{T}(\frac{1}{2}\text{tr}\, aP(s)+\frac{\beta^{2}}{2}\text{tr}P(s))ds\, m_{1}+\frac{1}{2}\beta^{2}\int_{t}^{T}\text{tr}\Sigma(s,m_{1})ds\,(m_{1})^{2}.\label{eq:5.16}
\end{equation}
 
Recalling that $y=\int_{\mathbb{R}^{n}}\xi m(\xi)d\xi$ and identifying quadratic terms in $\xi$ (within the integral) and in $y$ respectively, it follows easily that 
\begin{equation}
\left\{
\begin{array}{l}
\displaystyle\frac{dP}{dt}+PA+A^{*}P-PBR^{-1}B^{*}P+Q+\bar{Q}=0,\\
P(T)=Q_{T}+\bar{Q}_{T},\label{eq:5.17}\\
\end{array}
\right.
\end{equation}
and 
\begin{equation}
\left\{
\begin{array}{l}
\displaystyle\frac{d\Sigma}{dt}+\Sigma(A+\bar{A}m_{1}-BR^{-1}B^{*}P)+(A+\bar{A}m_{1}-BR^{-1}B^{*}P)^{*}\Sigma\label{eq:5.18}\\
\qquad\qquad -\displaystyle\Sigma BR^{-1}B^{*}\Sigma m_{1}+S^{*}\bar{Q}Sm_{1}-\bar{Q}S-S^{*}\bar{Q}+P\bar{A}+\bar{A}\,^{*}P=0;\\
\Sigma(T,m_{1})=S_{T}^{*}\bar{Q}_{T}S_{T}\, m_{1}-(S_{T}^{*}\bar{Q}_{T}+\bar{Q}_{T}S_{T}).
\end{array}
\right.
\end{equation}

We obtain formula (\ref{eq:5.13}) with the values of $P(t),\,\Sigma(t,m_{1}),\,\lambda(t,m_{1})$
given by equations (\ref{eq:5.17}), (\ref{eq:5.18}), (\ref{eq:5.16}). We next turn to the Master equation. The function $U(x,m,t)$ is given
by (\ref{eq:5.15}). Let us set $\Gamma(t,m_{1})=\frac{\partial\Sigma(t,m_{1})}{\partial m_{1}}$.
From (\ref{eq:5.18}) we obtain easily 

\begin{equation}
\left\{
\begin{array}{l}
\displaystyle\frac{d\Gamma}{dt}+\Gamma(A+\bar{A}m_{1}-BR^{-1}B^{*}(P+\Sigma m_{1}))+(A+\bar{A}m_{1}-BR^{-1}B^{*}(P+\Sigma m_{1}))^{*}\Gamma\label{eq:5.19}\\
\qquad\qquad\qquad\qquad\qquad\qquad\qquad\qquad\qquad\qquad +S^{*}\bar{Q}S-\Sigma BR^{-1}B^{*}\Sigma+\Sigma\bar{A}+\bar{A}\,^{*}\Sigma=0,\\
\Gamma(T,m_{1})=S_{T}^{*}\bar{Q}_{T}S_{T}.
\end{array}
\right.
\end{equation}

So we can write 
\begin{equation}
\label{eq:5.21}
\begin{array}{l}
U(x,m,t)=\displaystyle\frac{1}{2}x^{*}P(t)x+y^{*}\Sigma(t,m_{1})x+\frac{1}{2}y^{*}\Gamma(t,m_{1})y\\
\qquad\qquad\qquad\qquad+\displaystyle\int_{t}^{T}(\frac{1}{2}\text{tr}\, aP(s)+\frac{\beta^{2}}{2}\text{tr}P(s))ds\,+\frac{1}{2}\beta^{2}\int_{t}^{T}\text{tr}\Gamma(s,m_{1})ds\,(m_{1})^{2}+\beta^{2}\int_{t}^{T}\text{tr}\Sigma(s,m_{1})ds m_{1}.
\end{array}
\end{equation}

We want to check that this functional is the solution of the Master
equation (\ref{eq:2.8-1}). We rewrite the Master equation as follows under the linear quadratic setting

\begin{equation}
\label{eq:5.22}
\left\{
\begin{array}{l}
	-\displaystyle\frac{\partial U}{\partial t}-\frac{1}{2}\text{tr}\, aD^{2}U-\frac{1}{2}\beta^{2}\Delta U\\
	\qquad-\displaystyle\int_{\mathbb{R}^{n}}\Big(\frac{1}{2}\text{tr}\, aD_{\xi}^{2}\frac{\partial}{\partial m}U(x,m,t)(\xi)+\frac{1}{2}\beta^{2}\Delta_{\xi}\frac{\partial}{\partial m}U(x,m,t)(\xi)\Big)m(\xi)d\xi\\
	\qquad-\displaystyle\int_{\mathbb{R}^{n}}D_{\xi}\frac{\partial}{\partial m}U(x,m,t)(\xi).G(\xi,m,DU(\xi,m,t))m(\xi)d\xi\\
	\qquad-\displaystyle\frac{1}{2}\beta^{2}\int_{\mathbb{R}^{n}}\int_{\mathbb{R}^{n}}\sum_{i=1}^{n}\frac{\partial}{\partial\xi_{i}}\frac{\partial}{\partial\eta_{i}}(\frac{\partial^{2}U(x,m,t)}{\partial m^{2}}(\xi,\eta))m(\xi)m(\eta)d\xi d\eta-\beta^{2}\sum_{i=1}^{n}\frac{\partial}{\partial x_{i}}\left(\int_{\mathbb{R}^{n}}\frac{\partial}{\partial\xi_{i}}\frac{\partial U(x,m,t)}{\partial m}(\xi)m(\xi)d\xi\right)\\
	=\displaystyle\frac{1}{2}x^{*}(Q+\bar{Q})x-x^{*}(\bar{Q}S+S^{*}\bar{Q})y+\frac{1}{2}y^{*}S^{*}\bar{Q}Sy+y^{*}S^{*}\bar{Q}S\, xm_{1}\\
	\qquad-\displaystyle\frac{1}{2}(DU(x,m,t))^{*}BR^{-1}B^{*}DU(x,m,t)+(DU(x,m,t))^{*}(Ax+\bar{A}y)+\int_{\mathbb{R}^{n}}(DU(\xi,m,t))^{*}m(\xi)d\xi\,\bar{A}x,\\
	U(x,m,T)=\displaystyle\frac{1}{2}x^{*}(Q_{T}+\bar{Q}_{T})x-x^{*}(\bar{Q}_{T}S_{T}+S_{T}^{*}\bar{Q}_{T})y+\frac{1}{2}y^{*}S_{T}^{*}\bar{Q}_{T}S_{T}y+y^{*}S_{T}^{*}\bar{Q}_{T}S_{T}\, xm_{1}.\\
\end{array}
\right.
\end{equation}

We look for a solution of the form 
\begin{equation}
U(x,m,t)=\frac{1}{2}x^{*}P(t)x+y^{*}\Sigma(t,m_{1})x+\frac{1}{2}y^{*}\Gamma(t,m_{1})y+\mu(t,m_{1})\label{eq:5.24}
\end{equation}
with $\Sigma(t,m_{1}),\,\Gamma(t,m_{1})$ symmetric. We can calculate that
\begin{equation*}
\begin{split}
	DU(x,m,t)&=P(t)x+\Sigma(t,m_{1})y,\\
	D^{2}U(x,m,t)&=P(t),\\
	\displaystyle\frac{\partial}{\partial m}U(x,m,t)(\xi)&=(x^{*}\Sigma(t,m_{1})+y^{*}\Gamma(t,m_{1}))\xi+y^{*}\frac{\partial}{\partial m_{1}}\Sigma(t,m_{1})x\\
	&\qquad\qquad\qquad+\frac{1}{2}y^{*}\frac{\partial}{\partial m_{1}}\Gamma(t,m_{1})y+\frac{\partial}{\partial m_{1}}\mu(t,m_{1}),\\
	D_{\xi}\displaystyle\frac{\partial}{\partial m}U(x,m,t)(\xi)&=\Sigma(t,m_{1})x+\Gamma(t,m_{1})y,\\
	D_{\xi}^{2}\displaystyle\frac{\partial}{\partial m}U(x,m,t)(\xi)&=0,\\
	\displaystyle\frac{\partial^{2}}{\partial m^{2}}U(x,m,t)(\xi,\eta)&=\eta^{*}\Gamma(t,m_{1})\xi+x^{*}\frac{\partial}{\partial m_{1}}\Sigma(t,m_{1})(\xi+\eta)+y^{*}\displaystyle\frac{\partial}{\partial m_{1}}\Gamma(t,m_{1})(\xi+\eta)\\
	&\qquad\qquad\qquad+y^{*}\frac{\partial^{2}}{\partial m_{1}^{2}}\Sigma(t,m_{1})x+\frac{1}{2}y^{*}\frac{\partial^{2}}{\partial m_{1}^{2}}\Gamma(t,m_{1})y+\frac{\partial^{2}}{\partial m_{1}^{2}}\mu(t,m_{1}),\\
	\sum_{i=1}^{n}\displaystyle\frac{\partial}{\partial\xi_{i}}\frac{\partial}{\partial\eta_{i}}(\frac{\partial^{2}U(x,m,t)}{\partial m^{2}}(\xi,\eta))&=\text{tr }\Gamma(t,m_{1}),\\
	\sum_{i=1}^{n}\frac{\partial}{\partial x_{i}}\frac{\partial}{\partial\xi_{i}}(\frac{\partial}{\partial m}U(x,m,t)(\xi))&=\text{tr }\Sigma(t,m_{1}).
\end{split}
\end{equation*}

Substituting all these results in the Master equation (\ref{eq:5.22}) yields 
\begin{equation*}
\begin{array}{l}
	-(\displaystyle\frac{1}{2}x^{*}\frac{d}{dt}P(t)x+y^{*}\frac{d}{d}\Sigma(t,m_{1})x+\frac{1}{2}y^{*}\frac{d}{dt}\Gamma(t,m_{1})y+\frac{d}{dt}\mu(t,m_{1}))-(\frac{1}{2}\text{tr}\, aP(t)+\frac{\beta^{2}}{2}\text{tr}P(t))\\
	\qquad -(\Sigma(t,m_{1})x+\Gamma(t,m_{1})y)^{*}(Ay+\bar{A}ym_{1}-BR^{-1}B^{*}(Py+\Sigma ym_{1}))\\
	\qquad -\displaystyle\frac{\beta^{2}}{2}\text{tr }\Gamma(t,m_{1})m_{1}^{2}-\beta^{2}\text{tr }\Sigma(t,m_{1})m_{1}\\
	=\displaystyle\frac{1}{2}x^{*}(Q+\bar{Q})x-x^{*}(\bar{Q}S+S^{*}\bar{Q})y+\frac{1}{2}y^{*}S^{*}\bar{Q}Sy+y^{*}S^{*}\bar{Q}S\, xm_{1}\\
	\qquad-\displaystyle\frac{1}{2}(P(t)x+\Sigma(t,m_{1})y)^{*}BR^{-1}B^{*}(P(t)x+\Sigma(t,m_{1})y)\\
	\qquad+(P(t)x+\Sigma(t,m_{1})y)^{*}(Ax+\bar{A}y)+(P(t)y+\Sigma(t,m_{1})ym_{1})^{*}\bar{A}x.
\end{array}
\end{equation*}

Comparing coefficients, one checks easily that $P(t),\,\Sigma(t,m_{1}),\,\Gamma(t,m_{1})$
satisfy the equations (\ref{eq:5.17}), (\ref{eq:5.18}), (\ref{eq:5.19}). Hence $\Sigma(t,m_{1})$ is symmetric and $\Gamma(t,m_{1})=\frac{\partial\Sigma(t,m_{1})}{\partial m_{1}}.$
Also 
\begin{equation*}
	\frac{d}{dt}\mu(t,m_{1})+\frac{1}{2}\text{tr}\, aP(t)+\frac{\beta^{2}}{2}\text{tr}P(t)+\frac{\beta^{2}}{2}\text{tr }\Gamma(t,m_{1})m_{1}^{2}+\beta^{2}\text{tr }\Sigma(t,m_{1})m_{1}=0.
\end{equation*}

Therefore $\mu(t,m_{1})=\frac{\partial\lambda(t,m_{1})}{\partial m_{1}}$ and we recover the formula $U(x,m,t)=\frac{\partial V(m,t)}{\partial m}(x)$. We can state the following proposition 
\begin{prop}
\label{prop1} If we assume (\ref{eq:5.2}), then the solution of the Bellman equation (\ref{eq:5.8}) is given by formula (\ref{eq:5.13}) with $P(t),$$\Sigma(t,m_{1}),\lambda(t,m_{1})$
given respectively by equations (\ref{eq:5.17}), (\ref{eq:5.18}),
(\ref{eq:5.16}). The solution of the Master equation (\ref{eq:2.8-1}) is given by formula (\ref{eq:5.21}), in which $\Gamma(t,m_{1})$
is the solution of (\ref{eq:5.19}). 
\end{prop}

\subsection{\label{sub:Mean Field Type}SYSTEM OF STOCHASTIC HJB-FP EQUATIONS
FOR MEAN FIELD TYPE CONTROL }

We now consider the solution of the stochastic Fokker-Planck equation 
\begin{equation}
\left\{
\begin{array}{l}
\partial_{t}m+(\mathcal{A}^{*}m-\frac{1}{2}\beta^{2}\Delta m+\text{div}(G(x,m,DU)m))dt+\beta Dm db(t)=0,\label{eq:5.25}\\
m(x,0)=m_{0}(x),
\end{array}
\right.
\end{equation}
for the linear quadratic Mean Field type control problem, assuming the
initial condition $m_{0}$ is Gaussian with mean $\bar{x}_{0}$ and
covariance matrix $\Pi_{0}.$ If we call $m(t)=m(x,t)$ the solution,
then 
\[
m_{1}(t)=\int_{\mathbb{R}^{n}}m(\xi,t)d\xi=1,\: y(t)=\int_{\mathbb{R}^{n}}\xi m(\xi,t)d\xi.
\]

Next 
\[
DU(x,m(t),t)=P(t)x+\Sigma(t)y(t),\;\Sigma(t)=\Sigma(t,1),
\]
\[
G(x,m,DU(x,m(t),t))=(A-BR^{-1}B^{*}P(t))x+(\bar{A}-BR^{-1}B^{*}\Sigma(t))y(t).
\]

From (\ref{eq:5.25}) by using the test function $x$, and then integrating in $x,$ we get easily that 
\begin{equation}
\label{eq:5.26}
\left\{
\begin{array}{l}
	dy=(A+\bar{A}-BR^{-1}B^{*}P(t))y(t)dt-BR^{-1}B^{*}\Sigma(t)y(t)dt+\beta db(t),\\
	y(0)=\bar{x}_{0}.\\
\end{array}
\right.
\end{equation}

We define 
\[
u(x,t)=U(x,m(t),t)=\frac{1}{2}x^{*}P(t)x+x^{*}r(t)+s(t),
\]
with 
\begin{equation}
\label{add_1}
r(t)=\Sigma(t)y(t),\: s(t)=\frac{1}{2}y(t)^{*}\Gamma(t)y(t)+\mu(t),
\end{equation}
in which $\Gamma(t)=\Gamma(t,1)$ and $\mu(t)=\mu(t,1)$. An easy
calculation, by taking account of (\ref{eq:5.18}) and (\ref{eq:5.26}) yields 
\begin{equation}
\label{eq:5.27}
\left\{
\begin{array}{l}
	-dr=(A^{*}+\bar{A}^{*}-PBR^{-1}B^{*})rdt+(S^{*}\bar{Q}S-\bar{Q}S-S^{*}\bar{Q}+P\bar{A}+\bar{A}\,^{*}P)ydt-\beta\Sigma db(t),\\
	r(T)=(S_{T}^{*}\bar{Q}_{T}S_{T}\,{\bf 1}-S_{T}^{*}\bar{Q}_{T}-\bar{Q}_{T}S_{T})y(T),
\end{array}
\right.
\end{equation}
and we can rewrite (\ref{eq:5.26}) as 

\begin{equation}
\label{eq:5.28}
\left\{
\begin{split}
dy&=(A+\bar{A}-BR^{-1}B^{*}P(t))y(t)dt-BR^{-1}B^{*}r(t)dt+\beta db(t),\\
y(0)&=\bar{x}_{0},\\
\end{split}
\right.
\end{equation}
and the pair $(y(t),r(t))$ becomes the solution of a system of forward-backward SDE. Considering the fundamental matrix $\Phi_{P}(t,s)$ associated to the matrix $A+\bar{A}-BR^{-1}B^{*}P(t)$: 
\begin{equation*}
\left\{
\begin{array}{l}
	\displaystyle\frac{\partial\Phi_{P}(t,s)}{\partial t}=(A+\bar{A}-BR^{-1}B^{*}P(t))\Phi_{P}(t,s),\,\forall t>s,\\
	\Phi_{P}(s,s)=I,
\end{array}
\right.
\end{equation*}
then we can write 
\begin{equation*}
\begin{array}{l}
	r(t)=\displaystyle\Phi_{P}^{*}(T,t)(S_{T}^{*}\bar{Q}_{T}S_{T}\,{\bf 1}-S_{T}^{*}\bar{Q}_{T}-\bar{Q}_{T}S_{T})y(T)\\
	\qquad\qquad+\displaystyle\int_{t}^{T}\Phi_{P}^{*}(s,t)(S^{*}\bar{Q}S-\bar{Q}S-S^{*}\bar{Q}+P(s)\bar{A}+\bar{A}\,^{*}P(s))y(s)ds-\beta\int_{t}^{T}\Phi_{P}^{*}(s,t)\Sigma(s)db(s).
\end{array}
\end{equation*}
 This relation implies 
\begin{equation}
\begin{split}
	&r(t)=\Phi_{P}^{*}(T,t)(S_{T}^{*}\bar{Q}_{T}S_{T}\,{\bf 1}-S_{T}^{*}\bar{Q}_{T}-\bar{Q}_{T}S_{T}) \mathbb{E} ^{\mathcal{B}^{t}}y(T)\label{eq:5.29}\\
	&\qquad\qquad+\int_{t}^{T}\Phi_{P}^{*}(s,t)(S^{*}\bar{Q}S-\bar{Q}S-S^{*}\bar{Q}+P(s)\bar{A}+\bar{A}\,^{*}P(s)) \mathbb{E} ^{\mathcal{B}^{t}}y(s)ds.
\end{split}
\end{equation}
 In the system (\ref{eq:5.28}), (\ref{eq:5.29}) the external function
$\Sigma(s)$ does not appear anymore, but $r(t)$ is the solution
of an integral equation, instead of an backward SDE. Finally, from (\ref{add_1}), by taking differentiation and then integrating from $t$ to $T$, we have 
\begin{align}
s(t) & =\frac{1}{2}y(T)^{*}S_{T}^{*}\bar{Q}_{T}S_{T}y(T)+\int_{t}^{T}(\frac{1}{2}\text{tr}\, aP(s)+\frac{\beta^{2}}{2}\text{tr}P(s)+\beta^{2}\text{tr }\Sigma(s))ds \nonumber \\
&\qquad+ \int_{t}^{T}(\frac{1}{2}y(s)^{*}S^{*}QSy(s)-\frac{1}{2}r(s)^{*}BR^{-1}B^{*}r(s)+r(s)^{*}\bar{A}y(s))ds-\beta\int_{t}^{T}y(s)^{*}\Gamma(s)db(s). \label{eq:5.30-1}
\end{align}

Since $s(t)$ is adapted to $\mathcal{B}^{t}$, we can write 
\begin{align}
s(t) & =\int_{t}^{T}(\frac{1}{2}\text{tr}\, aP(s)+\frac{\beta^{2}}{2}\text{tr}P(s)+\beta^{2}\text{tr }\Sigma(s))ds+\frac{1}{2} \mathbb{E} ^{\mathcal{B}^{t}}y(T)^{*}S_{T}^{*}\bar{Q}_{T}S_{T}y(T)\label{eq:5.31-1}\\
&\qquad+  \mathbb{E} ^{\mathcal{B}^{t}}\int_{t}^{T}(\frac{1}{2}y(s)^{*}S^{*}QSy(s)-\frac{1}{2}r(s)^{*}BR^{-1}B^{*}r(s)+r(s)^{*}\bar{A}y(s))ds.\nonumber 
\end{align}

The results obtained contains that in Bensoussan et al. \cite{BSYY} as a special case when $\beta=0$.

\subsection{MEAN FIELD GAMES MASTER EQUATION }

The Master equation (\ref{eq:4.7}) under the LQ setting is
\begin{equation}
\label{eq:5.31}
\left\{
\begin{array}{l}
	-\displaystyle\frac{\partial U}{\partial t}-\frac{1}{2}\text{tr}\, aD^{2}U-\frac{1}{2}\beta^{2}\Delta U\\
	\qquad-\displaystyle\int_{\mathbb{R}^{n}}\Big(\frac{1}{2}\text{tr}\, aD_{\xi}^{2}\frac{\partial}{\partial m}U(x,m,t)(\xi)+\frac{1}{2}\beta^{2}\Delta_{\xi}\frac{\partial}{\partial m}U(x,m,t)(\xi)\Big)m(\xi)d\xi\\
	\qquad -\displaystyle\int_{\mathbb{R}^{n}}D_{\xi}\frac{\partial}{\partial m}U(x,m,t)(\xi).G(\xi,m,DU(\xi,m,t))m(\xi)d\xi\\
	\qquad-\displaystyle\frac{1}{2}\beta^{2}\int_{\mathbb{R}^{n}}\int_{\mathbb{R}^{n}}\sum_{i=1}^{n}\frac{\partial}{\partial\xi_{i}}\frac{\partial}{\partial\eta_{i}}(\frac{\partial^{2}U(x,m,t)}{\partial m^{2}}(\xi,\eta))m(\xi)m(\eta)d\xi d\eta-\beta^{2}\displaystyle\sum_{i=1}^{n}\frac{\partial}{\partial x_{i}}\left(\int_{\mathbb{R}^{n}}\frac{\partial}{\partial\xi_{i}}\frac{\partial U(x,m,t)}{\partial m}(\xi)m(\xi)d\xi\right)\\
	=\displaystyle\frac{1}{2}x^{*}(Q+\bar{Q})x-x^{*}\bar{Q}Sy+\frac{1}{2}y^{*}S^{*}\bar{Q}Sy\\
	\qquad -\displaystyle\frac{1}{2}(DU(x,m,t))^{*}BR^{-1}B^{*}DU(x,m,t)+(DU(x,m,t))^{*}(Ax+\bar{A}y),\\
	U(x,m,T)=\displaystyle\frac{1}{2}x^{*}(Q_{T}+\bar{Q}_{T})x-x^{*}\bar{Q}_{T}S_{T}y+\frac{1}{2}y^{*}S_{T}^{*}\bar{Q}_{T}S_{T}y.
\end{array}
\right.
\end{equation}

We look for a solution of the form 
\begin{equation}
U(x,m,t)=\frac{1}{2}x^{*}P(t)x+x^{*}\Sigma(t,m_{1})y+\frac{1}{2}y^{*}\Gamma(t,m_{1})y+\mu(t,m_{1}).\label{eq:5.33}
\end{equation}

Using the terminal condition, we cannot have a symmetric $\Sigma(t,m_{1})$ this time. 

We can also compute the derivatives
\begin{equation*}
\begin{split}
	DU(x,m,t)&=P(t)x+\Sigma(t,m_{1})y,\\
	D^2U(x,m,t)&=P(t),\\
	\frac{\partial}{\partial m}U(x,m,t)(\xi)&=x^{*}\Sigma(t,m_{1})\xi+y^{*}\Gamma(t,m_{1})\xi+x^{*}\frac{\partial}{\partial m_{1}}\Sigma(t,m_{1})y\\
	&\qquad+\frac{1}{2}y^{*}\frac{\partial}{\partial m_{1}}\Gamma(t,m_{1})y+\frac{\partial}{\partial m_{1}}\mu(t,m_{1}),\\
	D_{\xi}\frac{\partial}{\partial m}U(x,m,t)(\xi)&=\Sigma^{*}(t,m_{1})x+\Gamma(t,m_{1})y,\\
	D_{\xi}^{2}\frac{\partial}{\partial m}U(x,m,t)(\xi)&=0,\\
	\frac{\partial^{2}}{\partial m^{2}}U(x,m,t)(\xi,\eta)&=(x^{*}\frac{\partial}{\partial m_{1}}\Sigma(t,m_{1})+y^{*}\frac{\partial}{\partial m_{1}}\Gamma(t,m_{1}))(\xi+\eta)\\
	&\qquad +\eta^{*}\Gamma(t,m_{1})\xi+x^{*}\frac{\partial^{2}}{\partial m_{1}^{2}}\Sigma(t,m_{1})y+\frac{1}{2}y^{*}\frac{\partial^{2}}{\partial m_{1}^{2}}\Gamma(t,m_{1})y+\frac{\partial^{2}}{\partial m_{1}^{2}}\mu(t,m_{1}),\\
	\sum_{i=1}^{n}\frac{\partial}{\partial\xi_{i}}\frac{\partial}{\partial\eta_{i}}(\frac{\partial^{2}U(x,m,t)}{\partial m^{2}}(\xi,\eta))&=\text{tr}\Gamma(t,m_{1}),\\
	\sum_{i=1}^{n}\frac{\partial}{\partial x_{i}}\frac{\partial}{\partial\xi_{i}}\frac{\partial U(x,m,t)}{\partial m}(\xi)&=\text{tr}\Sigma(t,m_{1}).
\end{split}
\end{equation*}

We apply these formulas in the Master equation (\ref{eq:5.31}) to obtain
\begin{equation*}
\begin{array}{l}
	-(\displaystyle\frac{1}{2}x^{*}\frac{d}{dt}P(t)x+x^{*}\frac{d}{dt}\Sigma(t,m_{1})y+\frac{1}{2}y^{*}\frac{d}{dt}\Gamma(t,m_{1})y+\frac{d}{dt}\mu(t,m_{1}))-(\frac{1}{2}\text{tr} aP(t)+\frac{1}{2}\beta^{2}\text{tr} P(t))\\
	\qquad-(\Sigma^{*}(t,m_{1})x+\Gamma(t,m_{1})y)^{*}(Ay+\bar{A}ym_{1}-BR^{-1}B^{*}(Py+\Sigma ym_{1}))\\
	\qquad-\displaystyle\frac{\beta^{2}}{2}\text{tr}\Gamma(t,m_{1})m_{1}^{2}-\beta^{2}\text{tr}\Sigma(t,m_{1})m_{1}\\
	=\displaystyle\frac{1}{2}x^{*}(Q+\bar{Q})x-x^{*}\bar{Q}Sy+\frac{1}{2}y^{*}S^{*}\bar{Q}Sy\\
	\qquad -\displaystyle\frac{1}{2}(P(t)x+\Sigma(t,m_{1})y)^{*}BR^{-1}B^{*}(P(t)x+\Sigma(t,m_{1})y)\\
	\qquad+(P(t)x+\Sigma(t,m_{1})y)^{*}(Ax+\bar{A}y).
\end{array}
\end{equation*}
Comparing coefficients, we obtain 
\begin{equation}
\left\{
\begin{array}{l}
\displaystyle\frac{d}{dt}P(t)+PA+A^{*}P-PBR^{-1}B^{*}P+Q+\bar{Q}=0,\label{eq:5.38}\\
P(T)=Q_{T}+\bar{Q}_{T}.
\end{array}
\right.
\end{equation}
\begin{equation}
\left\{
\begin{array}{l}
\displaystyle\frac{d\Sigma}{dt}+\Sigma(A+\bar{A}m_{1}-BR^{-1}B^{*}P)+(A^{*}-PBR^{-1}B^{*})\Sigma-\Sigma BR^{-1}B^{*}\Sigma m_{1}-\bar{Q}S+P\bar{A}=0,\label{eq:5.39}\\
\Sigma(T,m_{1})=-\bar{Q}_{T}S_{T}.
\end{array}
\right.
\end{equation}
\begin{equation}
\left\{
\begin{array}{l}
\displaystyle\frac{d\Gamma}{dt}+\Gamma(A+\bar{A}m_{1}-BR^{-1}B^{*}(P+\Sigma m_{1}))+(A+\bar{A}m_{1}-BR^{-1}B^{*}(P+\Sigma m_{1}))^{*}\Gamma\label{eq:5.40}\\
\qquad+S^{*}\bar{Q}S-\Sigma BR^{-1}B^{*}\Sigma+\Sigma\bar{A}+\bar{A}\,^{*}\Sigma=0,\\
\Gamma(T,m_{1})=S_{T}^{*}\bar{Q}_{T}S_{T}.
\end{array}
\right.
\end{equation}
\begin{equation}
\left\{
\begin{array}{l}
\displaystyle\frac{d}{dt}\mu(t,m_{1})+\frac{1}{2}\text{tr}\, aP(t)+\frac{\beta^{2}}{2}\text{tr}P(t)+\frac{\beta^{2}}{2}\text{tr }\Gamma(t,m_{1})m_{1}^{2}+\beta^{2}\text{tr }\Sigma(t,m_{1})m_{1}=0,\label{eq:5.41}\\
\mu(T,m_{1})=0.
\end{array}
\right.
\end{equation}
 
\begin{rem}
The function $\Gamma(t,m_{1})$ is no more the derivative of $\Sigma(t,m_{1})$
with respect to $m_{1}.$ \end{rem}
\begin{prop}
\label{prop2} If we assume (\ref{eq:5.2}) then the solution of the Master equation (\ref{eq:5.31})
is given by formula (\ref{eq:5.33}), with $P(t)$ solution of (\ref{eq:5.38}), $\Sigma(t,m_{1})$ solution of (\ref{eq:5.39}), $\Gamma(t,m_{1})$ solution of (\ref{eq:5.40}) and $\mu(t,m_{1})$ solution of (\ref{eq:5.41}). 
\end{prop}

\subsection{SYSTEM OF STOCHASTIC HJB-FP EQUATIONS FOR MEAN FIELD GAMES }

We consider in the LQ case the system of stochastic HJB-FP equations for Mean Field games, namely (\ref{eq:4.3}) and (\ref{eq:4.5}). We first consider the FP equation, which is the same as for the Mean Field type control, namely (\ref{eq:5.25}). With the notation of
Section \ref{sub:Mean Field Type}, we have 
\begin{equation*}
\left\{
\begin{array}{l}
dy=(A+\bar{A}-BR^{-1}B^{*}P(t))y(t)dt-BR^{-1}B^{*}r(t)dt+\beta db(t),\\
y(0)=\bar{x}_{0};
\end{array}
\right.
\end{equation*}
with $r(t)=\Sigma(t)y(t)$ and $\Sigma(t)=\Sigma(t,1)$. We obtain
that $r(t)$ satisfies 
\begin{equation*}
\left\{
\begin{array}{l}
-dr=(A^{*}-PBR^{-1}B^{*})rdt+(\bar{Q}S-P\bar{A})ydt-\beta\Sigma db(t),\\
r(T)=-\bar{Q}_{T}S_{T}y(T);
\end{array}
\right.
\end{equation*}
and 
\begin{equation*}
u(x,t)=U(x,m(t),t)=\frac{1}{2}x^{*}P(t)x+x^{*}r(t)+s(t).
\end{equation*}

We have 
\begin{equation*}
s(t)=\frac{1}{2}y^{*}\Gamma(t)y+\mu(t).
\end{equation*}

Since $\Gamma(t),y(t),\mu(t)$ are solutions of equations identical
to the Mean Field type control, we again obtain 
\begin{align}
s(t) & =\frac{1}{2}y(T)^{*}S_{T}^{*}\bar{Q}_{T}S_{T}y(T)+\int_{t}^{T}(\frac{1}{2}\text{tr}\, aP(s)+\frac{\beta^{2}}{2}\text{tr}P(s)+\beta^{2}\text{tr }\Sigma(s))ds\nonumber \\
&\qquad+ \int_{t}^{T}(\frac{1}{2}y(s)^{*}S^{*}QSy(s)-\frac{1}{2}r(s)^{*}BR^{-1}B^{*}r(s)+r(s)^{*}\bar{A}y(s))ds-\beta\int_{t}^{T}y(s)^{*}\Gamma(s)db(s).\nonumber 
\end{align}

The result coincides with that in Bensoussan \cite{BSYY} when $\beta=0$.

\section{NASH EQUILIBRIUM FOR A FINITE NUMBER OF PLAYERS }

\subsection{THE PROBLEM FOR A FINITE NUMBER OF PLAYERS}

We consider here $N$ players. Each of them has a state $x^{i}(t)\in \mathbb{R}^{n}$.
We define the vector of states $x(t)\in \mathbb{R}^{nN}$ as 
\[
x(t)=(x^{1}(t),\cdots,x^{N}(t)).
\]

We denote $x_{l}^{i}(t),l=1,\cdots,n,$ the components of the state
$x^{i}(t).$ The states evolve according to the following dynamics 
\begin{equation}
\left\{
\begin{array}{l}
dx^{i}(s) =g(x,v^{i}(x))ds+\sigma(x^{i})dw^{i}(s)+\beta db(s),\; s>t,\label{eq:6.1}\\
x^{i}(t) =x^{i};
\end{array}
\right. 
\end{equation}
 in which $v^{i}(x)\in \mathbb{R}^{d}$ is the control of player $i$. The
processes $w^{i}(s)$ and $b(s)$ are independent standard Wiener
processes in $\mathbb{R}^{n}$. We set, as usual, $a(x^{i})=\sigma(x^{i})\sigma(x^{i})^{*}$.
It is convenient to denote $v(x)=(v^{1}(x),\cdots,v^{N}(x)).$ We
introduce the cost functional of each player 
\begin{equation}
J^{i}(x,t;v(\cdot))= \mathbb{E} [\int_{t}^{T}f(x(s),v^{i}(x(s))ds+h(x(T))]\label{eq:6.2}.
\end{equation}

 We notice that the trajectories and cost functionals are linked only through the states and not through the controls. Consider the Hamiltonian 
\begin{equation}
H(x,q)=\inf_{v}[f(x,v)+q.g(x,v)].\label{eq:6.3}
\end{equation}

Denote $\hat{v}(x,q)$ the minimizer in the Hamiltonian, then we set $G(x,q)=g(x,\hat{v}(x,q))$.
We next consider the system of PDEs, $i=1, \cdots, N,$ 
\begin{equation}
\left\{
\begin{array}{l}
-\displaystyle\frac{\partial u^{i}}{\partial t}-\frac{1}{2}\sum_{j=1}^{N}\text{tr}(a(x^{j})D_{x^{j}}^{2}u^{i})-\frac{\beta^{2}}{2}\text{tr}\sum_{j,k=1}^{N}D_{x^{j}x^{k}}^{2}u^{i}-\sum_{j\not=i}D_{x^{j}}u^{i}.G(x,D_{x^{j}}u^{j})=H(x,D_{x^{i}}u^{i}),\label{eq:6.4}\\
u^{i}(x,T)=h(x),
\end{array}
\right. 
\end{equation}
 and define the feedbacks 
\begin{equation}
\hat{v}^{i}(x)=\hat{v}(x,D_{x^{i}}u^{i}(x)),\label{eq:6.5}
\end{equation}
which form a Nash equilibrium for the differential game (\ref{eq:6.1}),
(\ref{eq:6.2}). This means first that 

\begin{equation}
u^{i}(x,t)=J^{i}(x,t;\hat{v}(\cdot)),\label{eq:6.6}
\end{equation}
 and if we use the notation 
\[
v(x)=(v^{i}(x),\bar{v}^{i}(x)),
\]
 in which $\bar{v}^{i}(x)$ represents all the components of $v(x)$,
except $v^{i}(x)$, to emphasize the control of player $i$, then 
\begin{equation}
u^{i}(x,t)\leq J^{i}(x,t;v^{i}(x),\bar{\hat{v}}^{i}(x)),\:\forall v^{i}(x)\label{eq:6.7}.
\end{equation}

We next apply this framework in the following case. Consider $f(x,m,v)$, $g(x,m,v),h(x,m)$ as in the previous sections (with $x\in \mathbb{R}^{n})$, and assume this time that the argument $m$ is no more an element of $L^{2}(\mathbb{R}^{n})$ but a probability measure in $\mathbb{R}^{n}$. We define 
\begin{align*}
f(x,v^{i}(x)) & =f(x^{i},\frac{1}{N-1}\sum_{j\not=i}\delta_{x^{j}},v^{i}(x)),\quad \, g(x,v^{i}(x))=g(x^{i},\frac{1}{N-1}\sum_{j\not=i}\delta_{x^{j}},v^{i}(x)),\\
h(x) & =h(x^{i},\frac{1}{N-1}\sum_{j\not=i}\delta_{x^{j}}),
\end{align*}
in which $x=(x^{1},\cdots,x^{N}).$ The Hamiltonian becomes 
\[
H(x^{i},\frac{1}{N-1}\sum_{j\not=i}\delta_{x^{j}},q)=\inf_{v}(f(x^{i},\frac{1}{N-1}\sum_{j\not=i}\delta_{x^{j}},v)+q.g(x^{i},\frac{1}{N-1}\sum_{j\not=i}\delta_{x^{j}},v)).
\]

Consider the optimal feedback $\hat{v}(x^{i},\frac{1}{N-1}\sum_{j\not=i}\delta_{x^{j}},q)$,
we have 
\[
G(x^{i},\frac{1}{N-1}\sum_{j\not=i}\delta_{x^{j}},q)=g(x^{i},\frac{1}{N-1}\sum_{j\not=i}\delta_{x^{j}},\hat{v}(x^{i},\frac{1}{N-1}\sum_{j\not=i}\delta_{x^{j}},q)).
\]

Hence the system of PDEs as in (\ref{eq:6.4}) becomes, for $i=1, \cdots, N,$ 
\begin{equation}
\label{eq:6.8}
\left\{
\begin{array}{l}
-\displaystyle\frac{\partial u^{i}}{\partial t}-\frac{1}{2}\sum_{j=1}^{N}\text{tr}(a(x^{j})D_{x^{j}}^{2}u^{i})-\frac{\beta^{2}}{2}\text{tr}\sum_{j,k=1}^{N}D_{x^{j}x^{k}}^{2}u^{i}-\sum_{j\not=i}D_{x^{j}}u^{i}.G(x^{i},\frac{1}{N-1}\sum_{k\not=j}\delta_{x^{k}},D_{x^{j}}u^{j})\\
\qquad=H(x^{i},\frac{1}{N-1}\displaystyle\sum_{j\not=i}\delta_{x^{j}},D_{x^{i}}u^{i}), \\
u^{i}(x,T)=h(x,\frac{1}{N-1}\displaystyle\sum_{j\not=i}\delta_{x^{j}}).\\ 
\end{array}
\right.
\end{equation}

We want to interpret (\ref{eq:6.8}) as a Master equation, which is
the way P.L. Lions has introduced the concept of Master equation. 

\subsection{\label{sub:DISCUSSION-ON-THE}DISCUSSION ON THE DERIVATIVE WITH RESPECT
TO A PROBABILITY MEASURE. }

For a functional $F(m)$ where $m$ is in $L^{2}(\mathbb{R}^{n})$, we
have defined the concept of derivative by simply using the concept of G\^ateaux differentiability. When $m$ represents the density of
a probability measure, the functional becomes a functional of a probability
measure, but of a special kind. Suppose that $F$ extends to a general
probability measure, then the concept of differentiability does not
extend. For instance if $x^{j}$ is a point in $\mathbb{R}^{n}$, neither can the concept
of differentiability be extended to $F(\delta_{x^{j}}),$ nor
more generally to $F(\frac{\sum_{j=1}^{K}\delta_{x^{j}}}{K}).$ 

Nevertheless, we may have the following situation: the functional
$F(m)$ is differentiable in $m$, and the function 
\begin{equation}
\Phi(x)=\Phi(x^{1},\cdots,x^{N})=F(\frac{\sum_{j=1}^{K}\delta_{x^{j}}}{K})\label{eq:6.100}
\end{equation}
is differentiable in $x$. Note that differentiability refers to two different set up, one with respect to arguments in $L^{2}(\mathbb{R}^{n})$ and the other
one, with respect to arguments in $\mathbb{R}^{nN}.$ We want to study the
link between these two concepts. 
\begin{prop}
\label{prop:6.1} Assume that $F(m)$ and $\Phi(x)$ are sufficiently
differentiable in their own sense, and that expressions below are well-defined. Then we
have the following relations 
\begin{equation}
\int_{\mathbb{R}^{n}}D_{\xi}\frac{\partial F(m)}{\partial m}(\xi).B(\xi)m(\xi)d\xi=\sum_{j=1}^{K}D_{x^{j}}\Phi(x).B(x^{j});\label{eq:6.101}
\end{equation}
\begin{equation}
\int_{\mathbb{R}^{n}}\Delta_{\xi}\frac{\partial F(m)}{\partial m}(\xi)m(\xi)d\xi+\int_{\mathbb{R}^{n}}\int_{\mathbb{R}^{n}}\text{tr}D_{\xi}D_{\eta}(\frac{\partial^{2}F(m)}{\partial m^{2}}(\xi,\eta))m(\xi)m(\eta)d\xi d\eta=\sum_{j,k=1}^{K}\text{tr}(D_{x^{j}}D_{x^{k}}\Phi(x));\label{eq:6.102}
\end{equation}
\begin{equation}
\int_{\mathbb{R}^{n}}\text{tr}(a(\xi)D_{\xi}^{2}(\frac{\partial F(m)}{\partial m}(\xi)))m(\xi)d\xi \approx \sum_{j=1}^{K}\text{tr}(a(x^{j})D_{x^{j}}^{2}\Phi(x)),\:\text{for large } K\label{eq:6.103}; 
\end{equation}
where $m = \frac{1}{K} \sum_{j=1}^K \delta_{x^j}$.
\end{prop}
\begin{proof}
We first comment how to understand these relations. The left hand
side makes sense for sufficiently smooth functionals $F(m).$ The
results are functionals of $m$, defined on $L^{2}(\mathbb{R}^{n}).$ Suppose that 
we can interpret them in the case $m$ is replaced by $\frac{\sum_{j=1}^{K}\delta_{x^{j}}}{K}.$
We obtained functions of $x=(x^{1},\cdots,x^{N})\in \mathbb{R}^{nN}.$ The
statement tells that these functions are identical to those on the
right hand side, in which $\Phi(x)$ is defined by (\ref{eq:6.100}).

For (\ref{eq:6.103}), it is only an approximation valid for large
$K.$ To illustrate, consider the particular case:
\begin{equation}
F(m)=\int_{\mathbb{R}^{n}}\int_{\mathbb{R}^{n}}\Gamma(u,v)m(u)m(v)dudv\label{eq:6.104}
\end{equation}
in which $\Gamma(u,v)=\Gamma(v,u)$ is twice continuously
differentiable. We have 
\begin{equation*}
\begin{split}
	\frac{\partial  F(m)}{\partial m}(\xi)&=2\int_{\mathbb{R}^{n}}\Gamma(\xi,v)m(v)dv,\\
	D_{\xi}\frac{\partial F(m)}{\partial m}(\xi)&=2\int_{\mathbb{R}^{n}}D_{\xi}\Gamma(\xi,v)m(v)dv,\\
	D_{\xi}^{2}(\frac{\partial  F(m)}{\partial m}(\xi))&=2\int_{\mathbb{R}^{n}}D_{\xi}^{2}\Gamma(\xi,v)m(v)dv,\\
	\frac{\partial ^{2}F(m)}{\partial m^{2}}(\xi,\eta)&=2\Gamma(\xi,\eta),\\
	D_{\xi}D_{\eta}\frac{\partial ^{2}F(m)}{\partial m^{2}}(\xi,\eta)&=2D_{\xi}D_{\eta}\Gamma(\xi,\eta).\\
\end{split}
\end{equation*}

Hence, we obtain
\begin{equation*}
\begin{split}
&\int_{\mathbb{R}^{n}}D_{\xi}\frac{\partial F(m)}{\partial m}(\xi).B(\xi)m(\xi)d\xi=2\int_{\mathbb{R}^{n}}\int_{\mathbb{R}^{n}}D_{u}\Gamma(u,v).B(u)m(u)m(v)dudv;\\
&\int_{\mathbb{R}^{n}}\Delta_{\xi}\frac{\partial F(m)}{\partial m}(\xi)m(\xi)d\xi+\int_{\mathbb{R}^{n}}\int_{\mathbb{R}^{n}}\text{tr}D_{\xi}D_{\eta}(\frac{\partial^{2}F(m)}{\partial m^{2}}(\xi,\eta))m(\xi)m(\eta)d\xi d\eta\\
&\qquad=2\int_{\mathbb{R}^{n}}\int_{\mathbb{R}^{n}}\Delta_{u}\Gamma(u,v)m(u)m(v)dudv+2\int_{\mathbb{R}^{n}}\int_{\mathbb{R}^{n}}\text{tr}D_{u}D_{v}\Gamma(u,v)m(u)m(v)dudv;\\
&\int_{\mathbb{R}^{n}}\text{tr}(a(\xi)D_{\xi}^{2}(\frac{\partial F(m)}{\partial m}(\xi))m(\xi)d\xi=2\int_{\mathbb{R}^{n}}\int_{\mathbb{R}^{n}}\text{tr}(a(u)D_{u}^{2}\Gamma(u,v))m(u)m(v)dudv.\\
\end{split}
\end{equation*}

We now apply these formulas with $m(x)=\frac{\sum_{j=1}^{K}\delta_{x^{j}}(x)}{K}$ which yields 
\begin{equation*}
\begin{split}
&\int_{\mathbb{R}^{n}}D_{\xi}\frac{\partial F(m)}{\partial m}(\xi).B(\xi)m(\xi)d\xi=\frac{2}{K^{2}}\sum_{j,k=1}^{K}D_{x^{j}}\Gamma(x^{j},x^{k}).B(x^{j});\\
&\int_{\mathbb{R}^{n}}\Delta_{\xi}\frac{\partial F(m)}{\partial m}(\xi)m(\xi)d\xi+\int_{\mathbb{R}^{n}}\int_{\mathbb{R}^{n}}\text{tr}D_{\xi}D_{\eta}(\frac{\partial^{2}F(m)}{\partial m^{2}}(\xi,\eta))m(\xi)m(\eta)d\xi d\eta\\
&\qquad=\frac{2}{K^{2}}\sum_{j,k=1}^{K}\Delta_{x^{j}}\Gamma(x^{j},x^{k})+\frac{2}{K^{2}}\sum_{j\not=k=1}^{K}\text{tr}(D_{x^{j}}D_{x^{k}}\Gamma(x^{j},x^{k}));
\\
&\int_{\mathbb{R}^{n}}\text{tr}(a(\xi)D_{\xi}^{2}(\frac{\partial F(m)}{\partial m}(\xi))m(\xi)d\xi=\frac{2}{K^{2}}\sum_{j,k=1}^{K}\text{tr}(a(x^{j})D_{x^{j}}^{2}\Gamma(x^{j},x^{k})).\\
\end{split}
\end{equation*}

We have to be careful in interpreting these formulas: $D_{x^{j}}\Gamma(x^{j},x^{k})$ represents the gradient with respect
to the first argument, even when $k=j$. We have the similar convention for $\Delta_{x^{j}}\Gamma(x^{j},x^{k})$ or $D_{x^{j}}^{2}\Gamma(x^{j},x^{k})$. Observe that under this particular case, $\Phi(x)=\frac{1}{K^{2}}\sum_{j,k=1}^{K}\Gamma(x^{j},x^{k})$. Hence, we have
\begin{equation*}
D_{x^{j}}\Phi(x)=\frac{2}{K^{2}}\sum_{k=1}^{K}D_{x^{j}}\Gamma(x^{j},x^{k}),
\end{equation*}
and therefore 
\[
\sum_{j=1}^{K}D_{x^{j}}\Phi(x).B(x^{j})=\frac{2}{K^{2}}\sum_{j=1}^{K}\sum_{k=1}^{K}D_{x^{j}}\Gamma(x^{j},x^{k}).B(x^{j}) ,
\]
which proves (\ref{eq:6.101}). 

With $\Phi(x)=\frac{1}{K^{2}}\sum_{j\not=k=1}^{K}\Gamma(x^{j},x^{k})+\frac{1}{K^{2}}\sum_{j=1}^{K}\Gamma(x^{j},x^{j})$, we consider next 
\[
D_{x^{j}}D_{x^{k}}\Phi(x)=\frac{2}{K^{2}}D_{x^{j}}D_{x^{k}}\Gamma(x^{j},x^{k}),\:\text{if}\: j\not=k,
\]
and 
\begin{align*}
\Delta_{x^{j}}\Phi(x) =\frac{2}{K^{2}}\sum_{k\not=j=1}^{K}\Delta_{x^{j}}\Gamma(x^{j},x^{k})+\frac{1}{K^{2}}\Delta_{x^{j}}\Gamma(x^{j},x^{j})=\frac{2}{K^{2}}\sum_{k=1}^{K}\Delta_{x^{j}}\Gamma(x^{j},x^{k}).
\end{align*}
Hence,
\begin{equation*}
\begin{split}
	\sum_{j,k=1}^{K}\text{tr}(D_{x^{j}}D_{x^{k}}\Phi(x))&=\sum_{j\not=k=1}^{K}\text{tr}(D_{x^{j}}D_{x^{k}}\Phi(x))+\sum_{j=1}^{K}\Delta_{x^{j}}\Phi(x)\\
	&=\frac{2}{K^{2}}\sum_{j\not=k=1}^{K}\text{tr}(D_{x^{j}}D_{x^{k}}\Gamma(x^{j},x^{k}))+\frac{2}{K^{2}}\sum_{j,k=1}^{K}\Delta_{x^{j}}\Gamma(x^{j},x^{k}) , 
\end{split}
\end{equation*}
and thus (\ref{eq:6.102}) is obtained. 

Finally, observe that  
\[
D_{x^{j}}^{2}\Phi(x)=\frac{2}{K^{2}}\sum_{k=1}^{K}D_{x^{j}}^{2}\Gamma(x^{j},x^{k}),
\]
and hence
\[
\sum_{j=1}^{K}\text{tr}(a(x^{j})D_{x^{j}}^{2}\Phi(x))=\frac{2}{K^{2}}\sum_{j,k=1}^{K}\text{tr}(a(x^{j})D_{x^{j}}^{2}\Gamma(x^{j},x^{k}))
\]
 which also implies (\ref{eq:6.103}). We note that (\ref{eq:6.103}) is exact under this particular case. 

We begin by proving (\ref{eq:6.101}) in the general case. Consider the probability measure defined by
\[
m(x,s)=\frac{1}{K}\sum_{j=1}^{K}\delta_{x^{j}(s)}(x),
\]
where $x^{j}(s)$ satisfies 
\begin{equation*}
\left\{
\begin{split}
&\frac{dx^{j}(s)}{ds}=B(x^{j}(s)),\\
&x^{j}(0) =x^{j}.
\end{split}
\right.
\end{equation*}

We have 
\[
m(x,0)=m(x)=\frac{\sum_{j=1}^{K}\delta_{x^{j}}(x)}{K}.
\]

The probability measure satisfies the degenerate
Fokker Planck equation in the weak sense
\[
\frac{\partial m}{\partial s}+\text{div}(mB)=0 ,
\]
and one checks easily that, from the differentiability of $F(m)$ 
\[
\frac{d}{ds}F(m(s))|_{s=0}=\int_{\mathbb{R}^{n}}D_{\xi}\frac{\partial F(m)}{\partial m}(\xi)B(\xi)m(\xi)d\xi.
\]
 On the other hand, by the definition of $\Phi(x)$, one has 
\[
F(m(s))=\Phi(x^{1}(s),\cdots,x^{K}(s))
\]
and thus 
\[
\frac{d}{ds}F(m(s))|_{s=0}=\sum_{j=1}^{K}D_{x^{j}}\Phi(x).B(x^{j})
\]
 and (\ref{eq:6.101}) is obtained. 

We now turn to (\ref{eq:6.102}). We again consider the probability
distribution $m(x,s)=\frac{1}{K}\sum_{j=1}^{K}\delta_{x^{j}(s)}(x)$,
with this time 
\[
x^{j}(s)=x^{j}+\beta b(s).
\]

We can check that the probability distribution $m(x,s)$ satisfies
the stochastic partial differential equation 
\begin{equation}
\label{eq:6.112}
\left\{
\begin{split}
&\partial_{s}m(x,s)-\frac{\beta^{2}}{2}\Delta m(x,s) ds+\beta Dm(x,s).db(s)=0,\\
&m(x,0)=\frac{1}{K}\sum_{j=1}^{K}\delta_{x^{j}}(x).
\end{split}
\right.
\end{equation}

Indeed, this is obtained by taking a test function and writing 
\[
\int\varphi(\xi)m(\xi,s)d\xi=\frac{1}{K}\sum_{j=1}^{K}\varphi(x^{j}(s)) .
\]
 We then expand the right hand side, by using It\^o's formula, and obtain 

\[
d\int\varphi(\xi)m(\xi,s)d\xi=\beta\frac{1}{K}\sum_{j=1}^{K}D\varphi(x^{j}(s))db(s)+\frac{\beta^{2}}{2}\frac{1}{K}\sum_{j=1}^{K}\Delta\varphi(x^{j}(s))ds.
\]

Hence we have 
\begin{align*}
d\int\varphi(\xi)m(\xi,s)d\xi & =\beta\int m(\xi,s)D\varphi(\xi).db(s)+\frac{\beta^{2}}{2}\int m(\xi,s)\Delta\varphi(\xi)d\xi ds\\
&=-\beta  \int\varphi(\xi)Dm(\xi,s).db(s)d\xi+\frac{\beta^{2}}{2}\int\varphi(\xi)\Delta m(\xi,s)d\xi ds
\end{align*}
 and (\ref{eq:6.112}) follows immediately. We next write 
\begin{equation*}
\begin{split}
	&dF(m(s))=\int_{\mathbb{R}^{n}}\frac{\partial F(m(s))}{\partial m}(\xi)[-\beta Dm(\xi,s).db(s)d\xi+\frac{\beta^{2}}{2}\Delta m(\xi,s)d\xi ds]\\
	&\qquad\qquad\qquad\qquad+\frac{\beta^{2}}{2}\int_{\mathbb{R}^{n}}\int_{\mathbb{R}^{n}}\frac{\partial^{2} F(m)}{\partial m^{2}}(\xi,\eta)Dm(\xi,s)Dm(\eta,s)d\xi d\eta ds.
\end{split}
\end{equation*}

It follows that 
\begin{equation*}
\begin{split}
	&dF(m(s))\\
	=&\beta\int_{\mathbb{R}^{n}}D_{\xi}\frac{\partial F(m(s))}{\partial m}(\xi)m(\xi,s).db(s)d\xi\\
	&\qquad+\frac{\beta^{2}}{2}\int_{\mathbb{R}^{n}}\int_{\mathbb{R}^{n}}D_{\xi}D_{\eta}\frac{\partial^{2} F(m)}{\partial m^{2}}(\xi,\eta)m(\xi,s)m(\eta,s)d\xi d\eta ds+\frac{\beta^{2}}{2}\int_{\mathbb{R}^{n}}\Delta_{\xi}\frac{\partial F(m(s))}{\partial m}(\xi)m(\xi,s)d\xi ds.
\end{split}
\end{equation*}

Hence
\begin{equation}
\label{eq:6.113}
	\frac{d}{ds} \mathbb{E} F(m(s))|_{s=0}=\frac{\beta^{2}}{2}\bigg(\int_{\mathbb{R}^{n}}\Delta_{\xi}\frac{\partial F(m)}{\partial m}(\xi)m(\xi)d\xi+\int_{\mathbb{R}^{n}}\int_{\mathbb{R}^{n}}D_{\xi}D_{\eta}\frac{\partial^{2} F(m)}{\partial m^{2}}(\xi,\eta)m(\xi)m(\eta)d\xi d\eta \bigg).
\end{equation}

On the other hand 
\[
F(m(s))=\Phi(x^{1}(s),\cdots,x^{K}(s))
\]
 and by It\^o's formula again 
\[
dF(m(s))=\beta\sum_{j=1}^{K}D_{x^{j}}\Phi(x(s)).db(s)+\frac{\beta^{2}}{2}\sum_{j,k}\text{tr}D_{x^{j}}D_{x^{k}}\Phi(x(s))ds,
\]
which implies 
\[
\frac{d}{ds} \mathbb{E}\Big[ F(m(s))\Big]\bigg|_{s=0}=\frac{\beta^{2}}{2}\sum_{j,k}\text{tr}D_{x^{j}}D_{x^{k}}\Phi(x).
\]
 Comparing with (\ref{eq:6.113}), we obtain the formula (\ref{eq:6.102}).
 
We finally prove (\ref{eq:6.103}). We have $m(x,s)=\frac{1}{K}\sum_{j=1}^{K}\delta_{x^{j}(s)}(x),$
with 

\[
x^{j}(s)=x^{j}+\sigma(x^{j}(s))dw^{j}(s)
\]
 in which the $w^{j}(s)$ are independent standard Wiener processes
in $\mathbb{R}^{n}.$ We can write 
\begin{align}
\label{eq:6.114}
d\int\varphi(\xi)m(\xi,s)d\xi =\frac{1}{K}\sum_{j=1}^{K}D\varphi(x^{j}(s)).\sigma(x^{j}(s))dw^{j}(s)+\frac{1}{2} & \int m(\xi,s)\text{tr}(a(\xi)D^{2}\varphi(\xi))d\xi ds.
\end{align}
 It is not possible, this time, to write a closed form equation for
$m(x,s).$ However, for $K$ large, since the processes $x^{j}(s)$
are independent, $m(x,s)$ is close to its average, which satisfies the Fokker-Planck equation in a weak sense:
\[
\frac{\partial m}{\partial s}=\frac{1}{2}\sum_{\alpha,\beta=1}^{n}\frac{\partial^{2}(a_{\alpha\beta}(\xi)m)}{\partial\xi_{\alpha}\partial\xi_{\beta}}.
\]
Next, consider then $F(m(s)).$ We check easily that
\begin{equation}
\frac{d}{ds}F(m(s))|_{s=0}=\frac{1}{2}\int_{\mathbb{R}^{n}}\text{tr}(a(\xi)D_{\xi}^{2}(\frac{\partial F(m)}{\partial m}(\xi))m(\xi)d\xi.\label{eq:6.115}
\end{equation}

Since $F(m(s))=\Phi(x^{1}(s),\cdots,x^{K}(s))$, we can write 
\[
F(m(s)) \approx  \mathbb{E} \Phi(x^{1}(s),\cdots,x^{K}(s)),
\]
and from It\^o's formula we get easily 
\[
\frac{d}{ds}F(m(s))|_{s=0} \approx \frac{1}{2}\sum_{j=1}^{K}\text{tr}(a(x^{j})D_{x^{j}}^{2}\Phi(x)),
\]
and the result (\ref{eq:6.103}) follows.
\end{proof}

\subsection{INTERPRETATION OF THE MASTER EQUATION FOR MEAN FIELD GAMES}

Consider equation (\ref{eq:4.7}) which we write as follows
\begin{equation}
\label{eq:6.20}
\left\{
\begin{split}
	&-\frac{\partial U}{\partial t}-\frac{1}{2}\text{tr}(a(x)D^{2}U)-\frac{1}{2}\beta^{2}\Delta U-\int_{\mathbb{R}^{n}}\Big(\frac{1}{2}\text{tr}a(\xi)D_{\xi}^{2}\frac{\partial U(x,m,t)}{\partial m}(\xi)+\frac{1}{2}\beta^{2}\Delta_{\xi}\frac{\partial U(x,m,t)}{\partial m}(\xi)\Big)m(\xi)d\xi\\
	&\qquad\qquad-\int_{\mathbb{R}^{n}}D_{\xi}\frac{\partial U(x,m,t)}{\partial m}(\xi).G(\xi,m,DU(\xi))\, m(\xi)d\xi-\beta^{2}\int_{\mathbb{R}^{n}}\text{tr}D_{x}D_{\xi}(\frac{\partial U(x,m,t)}{\partial m}(\xi))m(\xi)d\xi\\
	&\qquad\qquad-\frac{1}{2}\beta^{2}\int_{\mathbb{R}^{n}}\int_{\mathbb{R}^{n}}\text{tr}D_{\xi}D_{\eta}(\frac{\partial^{2}U(x,m,t)}{\partial m^{2}}(\xi,\eta))m(\xi)m(\eta)d\xi d\eta\\
	&\qquad=H(x,m,DU(x)),\\
	&U(x,m,T)=h(x,m).\\	
\end{split}
\right.
\end{equation}

We have the following proposition:
\begin{prop}
\label{prop5} If the solution of (\ref{eq:6.20}) $U(x,m,t)$ is
sufficiently smooth, then the functions 

\begin{equation}
u_{i}(x,t)=u_{i}(x^{1},\cdots,x^{N},t)=U(x^{i},\frac{\sum_{j\not=i}\delta_{x^{j}}(\cdot)}{N-1},t), \text{for } i=1, \cdots, N, \label{eq:6.21}
\end{equation}
 satisfy approximately the system of relations (\ref{eq:6.8}). 
\end{prop}
\begin{proof}
We take $x=x^{i}$ and $m=\frac{\sum_{j\not=i}\delta_{x^{j}}(\cdot)}{N-1}$,
and we make use of Proposition \ref{prop:6.1}. We have the equivalences 
\begin{equation*}
\begin{split}
	\frac{\partial U(x^{i},\frac{1}{N-1}\sum_{j\not=i}\delta_{x^{j}},t)}{\partial t}&=\frac{\partial u^{i}(x,t)}{\partial t};\\
	\frac{1}{2}\text{tr}(a(x^{i})D_{x^{i}}^{2}U(x^{i},\frac{1}{N-1}\sum_{j\not=i}\delta_{x^{j}},t))&=\frac{1}{2}\text{tr}(a(x^{i})D_{x^{i}}^{2}u^{i}(x,t));\\
	\frac{1}{2}\beta^{2}\Delta_{x^{i}}U(x^{i},\frac{1}{N-1}\sum_{j\not=i}\delta_{x^{j}},t)&=\frac{1}{2}\beta^{2}\Delta_{x^{i}}u^{i}(x,t);\\
	\int_{\mathbb{R}^{n}}\frac{1}{2}\text{tr}a(\xi)D_{\xi}^{2}\frac{\partial U(x^{i},m,t)}{\partial m}(\xi)\, m(\xi)d\xi&\sim\sum_{j\not=i}\frac{1}{2}\text{tr}(a(x^{j})D_{x^{j}}^{2}u^{i}(x,t));\\
	\int_{\mathbb{R}^{n}}D_{\xi}\frac{\partial U(x^{i},m,t)}{\partial m}(\xi).G(\xi,m,DU(\xi))\, m(\xi)d\xi&=\sum_{j\not=i}D_{x^{j}}u^{i}(x,t).G(x^{i},\frac{1}{N-1}\sum_{k\not=j}\delta_{x^{k}},D_{x^{j}}u^{j}(x,t));\\
	\frac{1}{2}\beta^{2}\Big(\int_{\mathbb{R}^{n}}\Delta_{\xi}\frac{\partial U(x^{i},m,t)}{\partial m}(\xi)m(\xi)d\xi+\int_{\mathbb{R}^{n}}&\int_{\mathbb{R}^{n}}\text{tr}D_{\xi}D_{\eta}(\frac{\partial^{2}U(x^{i},m,t)}{\partial m^{2}}(\xi,\eta))m(\xi)m(\eta)d\xi d\eta\Big)\\
	&=\frac{1}{2}\beta^{2}\sum_{j,k\not=i}\text{tr}(D_{x^{j}x^{k}}^{2}u^{j}(x,t));\\
	\beta^{2}\int_{\mathbb{R}^{n}}\text{tr}D_{x^{i}}D_{\xi}(\frac{\partial U(x^{i},m,t)}{\partial m}(\xi))m(\xi)d\xi&=\beta^{2}\sum_{j\not=i}\text{tr}(D_{x^{j}x^{i}}^{2}u^{j}(x,t));\\
	H(x^{i},m,D_{x^{i}}U(x^{i},m,t))&=H(x^{i},\frac{1}{N-1}\sum_{k\not=j}\delta_{x^{k}},D_{x^{i}}u^{i}(x^{i},t)).\\
\end{split}
\end{equation*}

Using these equivalences, we check easily that the Master equation
(\ref{eq:6.20}) with $x=x^{i}$ and $m=\frac{1}{N-1}\sum_{j\not=i}\delta_{x^{j}}$
reduces to the system (\ref{eq:6.8}) with functions $u^{i}(x,t)$
given by (\ref{eq:6.21}). It is an approximation, only because of
(\ref{eq:6.103}). The reduction is exact, when $\sigma(x)=0.$
\end{proof}

\section{APPLICATION TO SYSTEMIC RISK }

\subsection{THE PROBLEM }

We discuss here an application of \textit{Mean Field Games and Systemic Risk}
introduced by R. Carmona, J.P. Fouque, and L.H. Sun  \cite{CFS}. There
are $N$ players, and the state equations are given as follows:
\begin{equation}
\label{eq:7.1}
\left\{
\begin{array}{l}
dx^{i}(s)=\displaystyle\Big[\alpha(\frac{1}{N}\sum_{j=1}^{N}x^{j}(s)-x^{i}(s))+v^{i}(s)\Big]ds+\sigma dw^{i}(s)+\beta db(s),\; s>t,\\
x^{i}(t)=x^{i}.
\end{array}
\right. 
\end{equation}

Each player has the payoff 
\begin{equation}
\begin{split}
J^{i}(x,t;v(\cdot))&= \mathbb{E} \bigg[\int_{t}^{T}\{\frac{1}{2}(v^{i}(x(s)))^{2}-\lambda v^{i}(x(s))(\frac{1}{N}\sum_{j=1}^{N}x^{j}(s)-x^{i}(s))\label{eq:7.2}\\
&\qquad\qquad+\frac{\mu}{2}  (\frac{1}{N}\sum_{j=1}^{N}x^{j}(s)-x^{i}(s))^{2}\}ds+\frac{c}{2}(\frac{1}{N}\sum_{j=1}^{N}x^{j}(T)-x^{i}(T))^{2}\bigg].
\end{split}
\end{equation}

In \cite{CFS}, the players are banks lending and borrowing from each
other. The state $x^{i}$ represents the log-monetary reserve of bank
$i$. The term 
\[
\alpha(\frac{1}{N}\sum_{j=1}^{N}x^{j}(s)-x^{i}(s))=\frac{\alpha}{N}\sum_{j=1}^{N}(x^{j}(s)-x^{i}(s))
\]
 represents the sum of rates at which bank $i$ borrows or lends to
bank $j$, with $\alpha>0.$ The control $v^{i}$ represents the rate
of lending or borrowing to the central bank. The pay-off contains
a penalty term on the difference between each reserve level and the
average. The parameter $\lambda$ is positive, and the corresponding
term models an incentive to borrow ($v^{i}>0)$ when the reserve of
bank $i$ is smaller than the average, or to lend in the opposite
case. The players look for a Nash equilibrium.

\subsection{MEAN FIELD GAME APPROXIMATION }

The preceding problem is a particular case of the problem (\ref{eq:6.1}),
(\ref{eq:6.2}). So we can consider the mean field analogue. Using our notations 
\begin{equation}
\label{add_2}
	m_{1}=\int_{\mathbb{R}^{n}}m(\xi)d\xi=1,\: y=\int_{\mathbb{R}^{n}}\xi m(\xi)d\xi,
\end{equation}
we introduce the functions: 
\begin{equation}
\begin{split}
f(x,m,v)&=\frac{1}{2}v^{2}-\lambda v(y-v)+\dfrac{\mu}{2}(y-x)^{2},\\
g(x,m,v)&=\alpha(y-x)+v,\\
h(x,m)&=\frac{c}{2}(y-x)^{2}.\\
\end{split}
\end{equation}

It is then easy to check that 
\begin{equation*}
H(x,m,q)=\frac{\mu-\lambda^{2}}{2}(y-x)^{2}+q(\alpha+\lambda)(y-x)-\frac{1}{2}q^{2} , 
\end{equation*}
and
\begin{equation*}
G(x,m,q)=(\alpha+\lambda)(y-x)-q.
\end{equation*}

To preserve convexity, \cite{CFS} assume that $\mu-\lambda^{2}>0.$ The Master equation (\ref{eq:4.7}) under this setting becomes 
\begin{equation}
\label{eq:7.8}
\left\{
\begin{split}
	&-\frac{\partial U}{\partial t}-\frac{1}{2}(\sigma^{2}+\beta^{2})\frac{\partial^{2}U}{\partial x^{2}}-\frac{1}{2}(\sigma^{2}+\beta^{2})\int_{R}\frac{\partial^{2}}{\partial\xi^{2}}\frac{\partial U(x,m,t)}{\partial m}(\xi)m(\xi)d\xi\\
	&\qquad-\int_{R}\frac{\partial}{\partial\xi}\frac{\partial U(x,m,t)}{\partial m}(\xi)G(\xi,m,\frac{\partial}{\partial\xi}U(\xi,m,t))\, m(\xi)d\xi-\beta^{2}\int_{R}\frac{\partial}{\partial x}\frac{\partial}{\partial\xi}\frac{\partial U(x,m,t)}{\partial m}(\xi)m(\xi)d\xi\\
	&\qquad-\frac{1}{2}\beta^{2}\int_{R}\int_{R}\frac{\partial}{\partial\xi}\frac{\partial}{\partial\eta}\frac{\partial^{2}U(x,m,t)}{\partial m^{2}}(\xi,\eta)m(\xi)m(\eta)d\xi d\eta\\
	&=\frac{\mu-\lambda^{2}}{2}(y-x)^{2}+(\alpha+\lambda)(y-x)\frac{\partial U(x,m,t)}{\partial x}-\frac{1}{2}(\frac{\partial U(x,m,t)}{\partial x})^{2},\\
	&U(x,m,T)=\frac{c}{2}(y-x)^{2}.
\end{split}
\right.
\end{equation}

\subsection{SOLUTION}

To get an explicit solution of (\ref{eq:7.8}), we postulate that
\begin{equation}
U(x,m,t)=\frac{1}{2}(x-y)^{2}P(t)+R(m_{1},t).
\end{equation}

Clearly, we have
\begin{equation}
P(T)=c,\quad R(m_{1},T)=0
\end{equation}
from the terminal condition. An easy calculation, by using (\ref{eq:5.38}) to (\ref{eq:5.41}), leads to
\begin{equation*}
	\frac{dP}{dt}-2(\alpha+\lambda)P-P^{2}+\mu-\lambda^{2}=0
\end{equation*} 
and 
\begin{equation*}
	R(m_{1},t)=\frac{1}{2}(\sigma^{2}+\beta^{2}(m_{1}-1)^{2}) \int_{t}^{T}P(s)ds.
\end{equation*}

The corresponding FP equation becomes
\begin{equation*}
\left\{
\begin{split}
	&\partial_{t}m+(-\frac{1}{2}(\sigma^{2}+\beta^{2})\frac{\partial^{2}m}{\partial x^{2}}+(\alpha+\lambda+P(t))\frac{\partial}{\partial x}((y-x)m))dt+\beta\frac{\partial m}{\partial x}db(t)=0,\\
	&m(x,0)=m_{0}(x).
\end{split}
\right.
\end{equation*}

From the definition of $y$ in (\ref{add_2}), we have
\begin{equation}
y(t)=y_{0}+\beta b(t).
\end{equation}

The solution of the stochastic HJB equation (\ref{eq:4.3}) is then 

\begin{equation}
u(x,t)=u(x,m(t),t)=\frac{1}{2}(x-y_{0}-\beta b(t))^{2}P(t)+R(t) , \label{eq:7.14}
\end{equation}
 with $R(t)=R(1,t).$ We have $B(x,t)=(x-y(t)).$ We here recover the results in \cite{CFS}. 

\end{document}